\documentclass[a4paper,11pt]{amsart}
\usepackage{mathrsfs}
\usepackage{cases}
\usepackage{amsfonts}
\usepackage{graphicx}
\usepackage{amsmath}
\usepackage{amssymb}
\usepackage[all]{xypic}
\usepackage[all]{xy}
\usepackage{color}
\usepackage{colordvi}
\usepackage{multicol}
\begin{document}
\input xy
\xyoption{all}

\newcommand{\dimv}{\operatorname{\underline{dim}}\nolimits}
\newcommand{\Aus}{\operatorname{Aus}\nolimits}
\renewcommand{\mod}{\operatorname{mod}\nolimits}
\newcommand{\proj}{\operatorname{proj.}\nolimits}
\newcommand{\inj}{\operatorname{inj.}\nolimits}
\newcommand{\rad}{\operatorname{rad}\nolimits}
\newcommand{\res}{\operatorname{res}\nolimits}
\newcommand{\soc}{\operatorname{soc}\nolimits}
\newcommand{\Char}{\operatorname{char}\nolimits}
\newcommand{\Mod}{\operatorname{Mod}\nolimits}
\newcommand{\R}{\operatorname{R}\nolimits}
\newcommand{\End}{\operatorname{End}\nolimits}
\newcommand{\gsub}{\operatorname{gsub}\nolimits}
\newcommand{\Ht}{\operatorname{Ht}\nolimits}
\newcommand{\ind}{\operatorname{ind}\nolimits}
\newcommand{\rep}{\operatorname{rep}\nolimits}
\newcommand{\Ext}{\operatorname{Ext}\nolimits}
\newcommand{\Tor}{\operatorname{Tor}\nolimits}
\newcommand{\Hom}{\operatorname{Hom}\nolimits}
\newcommand{\Pic}{\operatorname{Pic}\nolimits}
\newcommand{\Coker}{\operatorname{Coker}\nolimits}
\newcommand{\GL}{\operatorname{GL}\nolimits}

\newcommand{\Gproj}{\operatorname{Gproj}\nolimits}
\newcommand{\irr}{\operatorname{irr}\nolimits}
\newcommand{\RHom}{\operatorname{RHom}\nolimits}
\renewcommand{\deg}{\operatorname{deg}\nolimits}
\renewcommand{\Im}{\operatorname{Im}\nolimits}
\newcommand{\Ker}{\operatorname{Ker}\nolimits}
\newcommand{\Aut}{\operatorname{Aut}\nolimits}
\newcommand{\Id}{\operatorname{Id}\nolimits}
\newcommand{\Qcoh}{\operatorname{Qch}\nolimits}
\newcommand{\CM}{\operatorname{CM}\nolimits}
\newcommand{\Cp}{\operatorname{Cp}\nolimits}
\newcommand{\coker}{\operatorname{Coker}\nolimits}
\renewcommand{\dim}{\operatorname{dim}\nolimits}
\renewcommand{\div}{\operatorname{div}\nolimits}
\newcommand{\Ab}{{\operatorname{Ab}\nolimits}}
\renewcommand{\Vec}{{\operatorname{Vec}\nolimits}}
\newcommand{\pd}{\operatorname{proj.dim}\nolimits}
\newcommand{\id}{\operatorname{inj.dim}\nolimits}
\newcommand{\Gd}{\operatorname{G.dim}\nolimits}
\newcommand{\gldim}{\operatorname{gl.dim}\nolimits}
\newcommand{\sdim}{\operatorname{sdim}\nolimits}
\newcommand{\add}{\operatorname{add}\nolimits}
\newcommand{\pr}{\operatorname{pr}\nolimits}
\newcommand{\oR}{\operatorname{R}\nolimits}
\newcommand{\oL}{\operatorname{L}\nolimits}
\newcommand{\Gr}{\operatorname{Gr}\nolimits}
\newcommand{\Perf}{{\mathfrak Perf}}

\newcommand{\cc}{{\mathcal C}}
\newcommand{\ce}{{\mathcal E}}
\newcommand{\cs}{{\mathcal S}}
\newcommand{\cf}{{\mathcal F}}
\newcommand{\cx}{{\mathcal X}}
\newcommand{\ct}{{\mathcal T}}
\newcommand{\cu}{{\mathcal U}}
\newcommand{\cv}{{\mathcal V}}
\newcommand{\cn}{{\mathcal N}}
\newcommand{\ch}{{\mathcal H}}
\newcommand{\ca}{{\mathcal A}}
\newcommand{\cb}{{\mathcal B}}
\newcommand{\ci}{{\mathcal I}}
\newcommand{\cj}{{\mathcal J}}
\newcommand{\cm}{{\mathcal M}}
\newcommand{\cp}{{\mathcal P}}
\newcommand{\cq}{{\mathcal Q}}
\newcommand{\cg}{{\mathcal G}}
\newcommand{\cw}{{\mathcal W}}
\newcommand{\co}{{\mathcal O}}
\newcommand{\cd}{{\mathcal D}}
\newcommand{\ck}{{\mathcal K}}
\newcommand{\calr}{{\mathcal R}}
\newcommand{\ol}{\overline}
\newcommand{\ul}{\underline}
\newcommand{\st}{[1]}
\newcommand{\ow}{\widetilde}
\renewcommand{\P}{\mathbf{P}}
\newcommand{\pic}{\operatorname{Pic}\nolimits}
\newcommand{\Spec}{\operatorname{Spec}\nolimits}
\newtheorem{theorem}{Theorem}[section]
\newtheorem{acknowledgement}[theorem]{Acknowledgement}
\newtheorem{algorithm}[theorem]{Algorithm}
\newtheorem{axiom}[theorem]{Axiom}
\newtheorem{case}[theorem]{Case}
\newtheorem{claim}[theorem]{Claim}
\newtheorem{conclusion}[theorem]{Conclusion}
\newtheorem{condition}[theorem]{Condition}
\newtheorem{conjecture}[theorem]{Conjecture}
\newtheorem{construction}[theorem]{Construction}
\newtheorem{corollary}[theorem]{Corollary}
\newtheorem{criterion}[theorem]{Criterion}
\newtheorem{definition}[theorem]{Definition}
\newtheorem{example}[theorem]{Example}
\newtheorem{exercise}[theorem]{Exercise}
\newtheorem{lemma}[theorem]{Lemma}
\newtheorem{notation}[theorem]{Notation}
\newtheorem{problem}[theorem]{Problem}
\newtheorem{proposition}[theorem]{Proposition}
\newtheorem{remark}[theorem]{Remark}
\newtheorem{solution}[theorem]{Solution}
\newtheorem{summary}[theorem]{Summary}
\newtheorem*{thm}{Theorem}

\def \bp{{\mathbf p}}
\def \bA{{\mathbf A}}
\def \bL{{\mathbf L}}
\def \bF{{\mathbf F}}
\def \bS{{\mathbf S}}
\def \bC{{\mathbf C}}

\def \Z{{\Bbb Z}}
\def \F{{\Bbb F}}
\def \C{{\Bbb C}}
\def \N{{\Bbb N}}
\def \Q{{\Bbb Q}}
\def \G{{\Bbb G}}
\def \P{{\Bbb P}}
\def \K{{\Bbb K}}
\def \E{{\Bbb E}}
\def \A{{\Bbb A}}
\def \BH{{\Bbb H}}
\def \T{{\Bbb T}}

\title[Cohen-Macaulay Auslander algebras]{Desingularization of quiver Grassmannians for Gentle algebras}
\author[Chen]{Xinhong Chen}
\address{Department of Mathematics, Southwest Jiaotong University, Chengdu 610031, P.R.China}
\email{chenxinhong@swjtu.edu.cn}

\author[Lu]{Ming Lu$^\dag$}
\address{Department of Mathematics, Sichuan University, Chengdu 610064, P.R.China}
\email{luming@scu.edu.cn}
\thanks{$^\dag$ Corresponding author}


\begin{abstract}
Following \cite{CFR13}, a desingularization of arbitrary quiver Grassmannians for finite dimensional Gorenstein projective modules of $1$-Gorenstein gentle algebras is constructed in terms of quiver Grassmannians for their Cohen-Macaulay Auslander algebras.
\end{abstract}

\maketitle

\section{Introduction}
\subsection{Gorenstein projective mdules} The concept of Gorenstein projective modules over any ring can be dated back to \cite{AB}, where Auslander and Bridger introduced the modules of $G$-dimension zero over two-sided noetherian rings, and formed by Enochs and Jenda \cite{EJ}. This class of modules satisfy some good stable properties, becomes a main ingredient in the relative homological algebra, and widely used in the representation theory of algebras and algebraic geometry, see e.g. \cite{AB,AR2,EJ,Bu,Ha1,Be}. It also plays as an important tool to study the representation theory of Gorenstein algebra, see e.g. \cite{AR2,Bu,Ha1}.

\subsection{Gorenstein algebras} Gorenstein algebra $\Lambda$, where by definition $\Lambda$ has finite injective dimension both as a left and a right $\Lambda$-module, is inspired from commutative ring theory.
A fundamental result of Buchweitz \cite{Bu} and Happel \cite{Ha1} states that for a Gorenstein algebra $\Lambda$, the singularity category is triangle equivalent to the stable category of the Gorenstein projective (also called (maximal) Cohen-Macaulay) $\Lambda$-modules, which generalized Rickard's result \cite{Ri} on self-injective algebras.

\subsection{Cohen-Macaulay Auslander algebras} For any Artin algebra $\Lambda$, denote by $\Gproj\Lambda$ its subcategory of Gorenstein projective modules. If $\Gproj\Lambda$ has only finitely many isoclasses of indecomposable objects, then $\Lambda$ is called CM-\emph{finite}. In this case, inspired by the definition of Auslander algebra, the Cohen-Macaulay Auslander algebra (also called the relative Auslander algebra) is defined to be $\End_\Lambda(\bigoplus_{i=1}^n E_i)^{op}$, where $E_1,\dots,E_n$ are all pairwise non-isomorphic indecomposable Gorenstein projective modules, \cite{Be1,Be,LZ}.
A CM-finite algebra $\Lambda$ is Gorenstein if and only if $\gldim \Aus(\Gproj\Lambda)<\infty$, \cite{LZ,Be}.
Pan proves that for any two Gorenstein Artin algebras $A$ and $B$ which are CM-finite, if $A$ and $B$ are derived equivalent, then their Cohen-Macaulay Auslander algebras are also
derived equivalent \cite{P}.

\subsection{Gentle algebras} As an important class of Gorenstein algebras \cite{GR}, gentle algebras were introduced in \cite{AS} as appropriate context for the investigation of algebras derived equivalent to hereditary algebras of type $\tilde{\A}_n$. The gentle algebras which are trees are precisely the algebras derived equivalent to hereditary algebras of type $\A_n$, see \cite{AH}.
It is interesting to notice that the class of gentle algebras is closed under derived equivalence, \cite{SZ}. For singularity categories of gentle algebras, Kalck determines their singularity category by finite products of $d$-cluster categories of type $\A_1$ \cite{Ka}. From \cite{Ka}, it is easy to see that Gentle algebras are CM-finite, which inspires us to study the properties of the Cohen-Macaulay Auslander algebras of Gentle algebras \cite{CL2}. Moreover, many important algebras are gentle, such as tilted algebras of type $\A_n$, algebras derived equivalent to $\A_n$-configurations of projective lines \cite{Bur} and also the cluster-tilted algebras of type $\A_n$ \cite{BMR}, $\tilde{\A}_n$ \cite{ABCP}.

\subsection{Quiver Grassmannians}
Quiver Grassmannians were first introduced by Crawley-Boevey \cite{C-B89}, Schofield \cite{Sch} to study the generic properties of quiver representations. They are projective varieties parametrizing subrepresentations of a quiver representation and Reineke shows in \cite{Rei} that every projective variety can be realized as a quiver Grassmannian. It was observed in \cite{CC} that these varieties play an important role in the additive categorification of (quantum) cluster algebra theory \cite{FZ} since cluster variables can be described in terms of the Euler characteristic of quiver Grassmannians, cf. for example \cite{CC,CK,DWZ,Na}. Subsequently, specific classes of quiver Grassmannians were studied by several authors, in particular, the varieties of subrepresentations of exceptional quiver representations since they are smooth projective varieties, see e.g. \cite{CR}.

In \cite{CFR12,CFR13a}, Cerulli-Feigin-Reineke initiated a systematic study of (singular) quiver Grassmannians of Dynkin quivers, starting from the surprising observation that the type $\A$ degenerate flag varieties studies in \cite{Fei1,Fei2,FF} are of this form. An important aspect of their work is the construction of desingularizations for the Dynkin type, namely a desingularization of arbitrary quiver Grassmannians for representations of Dynkin quivers is constructed in terms of quiver Grassmannians for an algebra derived equivalent to the Auslander algebra of the quiver, which is done in \cite{CFR13} generalizing \cite{FF} for the type $\A$ degenerate flag varieties. After that, they link these desingularizations to a construction by Hernandez-Leclerc \cite{HL}, which has been generalizeed by Leclerc-Plamondon \cite{LP} and further generalized by Keller-Scherotzke \cite{KS1,KS2}. The desingularization constructed in \cite{KS2} is the desingularization map for graded quiver varieties introduced by Nakajima \cite{Na1,Na}, and this generalized \cite{CFR13} to much more general situations and in particular for all modules over the repetitive algebra of an arbitrary iterated tilted algebra of Dynkin type.

\subsection{Main result}In this paper, we consider the quiver Grassmannians of arbitrary Gorenstein projective modules over $1$-Gorenstein gentle algebras $\Lambda$. Inspired by \cite{CFR13}, we use quiver Grassmannians over the Cohen-Macaulay Auslander algebra $\Gamma$ of $\Lambda$ to construct a desingularization of the quiver Grassmannians of arbitrary Gorenstein projective modules over $\Lambda$.

The paper is organized as follows. In section 2, we collect some standard material on gentle algebras, Gorenstein algebras and quiver Grassmannians. In Section 3, we define the key functor $\Phi:\mod \Lambda \rightarrow \mod \Gamma$, and prove that for any Gorenstein projective $\Lambda$-module $M$, the quiver Grassmannians of $\Phi(M)$ is smooth with irreducible and equidimensional connected components. In Section 4, we use the technique of \cite{CFR13} and the results of the previous sections to construct the desingularizations of quiver Grassmannians. Finally, in Section 5 we close with some examples of desingularizations.

\vspace{0.2cm} \noindent{\bf Acknowledgments.}
This work is inspired by some discussions with the first author's supervisor Bin Zhang. The authors thank him very much.

The first author(X. Chen) was supported by the Fundamental Research Funds for the Central Universities A0920502051411-45.

The corresponding author(M. Lu) was supported by the National Natural Science Foundation of China (Grant No. 11401401).

\section{Preliminaries}
Let $K$ be an algebraically closed field. For a $K$-algebra, we always mean a basic finite dimensional associative $K$-algebra. For an additive category $\ca$, we denote by $\ind \ca$ the isomorphism classes of indecomposable objects in $\ca$.

Let $Q$ be a quiver and $\langle I\rangle$ an admissible ideal in the path algebra $KQ$ which is generated by a set of relations $I$. Denote by $(Q,I)$ the \emph{associated bound quiver}. For any arrow $\alpha$ in $Q$ we denote by $s(\alpha)$ its starting vertex and by $t(\alpha)$ its ending vertex. An \emph{oriented path} (or path for short) $p$ in $Q$ is a sequence $p=\alpha_1\alpha_2\dots \alpha_r$ of arrows $\alpha_i$ such that $t(\alpha_i)=s(\alpha_{i-1})$ for all $i=2,\dots,r$.
\subsection{Gentle algebras}
We first recall the definition of special biserial algebras and of gentle algebras.
\begin{definition}[\cite{SW}]
The pair $(Q,I)$ is called \emph{special biserial} if it satisfies the following conditions.
\begin{itemize}
\item Each vertex of $Q$ is starting point of at most two arrows, and end point of at most two arrows.
\item For each arrow $\alpha$ in $Q$ there is at most one arrow $\beta$ such that $\alpha\beta\notin I$, and at most one arrow $\gamma$ such that $\gamma\alpha\notin I$.
\end{itemize}
\end{definition}
\begin{definition}[\cite{AS}]
The pair $(Q,I)$ is called \emph{gentle} if it is special biserial and moreover the following holds.
\begin{itemize}
\item The set $I$ is generated by zero-relations of length $2$.
\item For each arrow $\alpha$ in $Q$ there is at most one arrow $\beta$ with $t(\beta)=s(\alpha)$ such that $\alpha\beta\in I$, and at most one arrow $\gamma$ with $s(\gamma)=t(\alpha)$ such that $\gamma\alpha\in I$.
\end{itemize}

\end{definition}
A finite dimensional algebra $A$ is called \emph{special biserial} (resp., \emph{gentle}), if it has a presentation as $A=KQ/\langle I\rangle$ where $(Q,I)$ is special biserial (resp., gentle).

\subsection{Singularity categories and Gorenstein algebras}
Let $\Gamma$ be a finite dimensional $K$-algebra. Let $\mod \Gamma$ be the category of finitely generated left $\Gamma$-modules. For an arbitrary $\Gamma$-module $_\Gamma X$, we denote by $\pd_\Gamma X$ (resp. $\id_\Gamma X$) the projective dimension (resp. the injective dimension) of the module $_\Gamma X$. A $\Gamma$-module $G$ is \emph{Gorenstein projective}, if there is an exact sequence $$P^\bullet:\cdots \rightarrow P^{-1}\rightarrow P^0\xrightarrow{d^0}P^1\rightarrow \cdots$$ of projective $\Gamma$-modules, which stays exact under $\Hom_\Gamma(-,\Gamma)$, and such that $G\cong \Ker d^0$. We denote by $\Gproj(\Gamma)$ the subcategory of Gorenstein projective $\Gamma$-modules.

\begin{definition}[\cite{AR1,AR2}]
A finite dimensional algebra $\Gamma$ is called a Gorenstein algebra if $\Gamma$ satisfies $\id \Gamma_\Gamma<\infty$ and $\id_\Gamma \Gamma<\infty$.
\end{definition}

Observe that for a Gorenstein algebra $\Gamma$, we have $\id _\Gamma\Gamma=\id \Gamma_\Gamma$, \cite[Lemma 6.9]{Ha1}; the common value is denoted by $\Gd \Gamma$. If $\Gd \Gamma\leq d$, we say that $\Gamma$ is $d$-Gorenstein.

\begin{theorem}[\cite{Bu,EJ}]\label{theorem characterize of gorenstein property}
Let $\Gamma$ be an artin algebra and let $d\geq0$. Then the following statements are equivalent:

(1) the algebra $\Gamma$ is $d$-Gorenstein;

(2) $\Gproj(\Gamma)=\Omega^d(\mod \Gamma)$, where $\Omega$ is the syzygy functor.

In this case, A module $G$ is Gorenstein projective if and only if there is an exact sequence $0\rightarrow G\rightarrow P^0\rightarrow P^1\rightarrow \cdots$ with each
$P^i$ projective.
\end{theorem}

For an algebra $\Gamma$, the \emph{singularity category} of $\Gamma$ is defined to be the quotient category $D_{sg}^b(\Gamma):=D^b(\Gamma)/K^b(\proj \Gamma)$ \cite{Bu,Ha1,Or1}. Note that $D_{sg}^b(\Gamma)$ is zero if and only if $\gldim \Gamma<\infty$ \cite{Ha1}.

\begin{theorem}\cite{Bu,Ha1}
Let $\Gamma$ be a Gorenstein algebra. Then $\Gproj (\Gamma)$ is a Frobenius category with the projective modules as the projective-injective objects. The stable category $\underline{\Gproj}(\Gamma)$ is triangle equivalent to the singularity category $D^b_{sg}(\Gamma)$ of $\Gamma$.
\end{theorem}

An algebra is of \emph{finite Cohen-Macaulay type}, or simply, \emph{CM-finite}, if there are only finitely many isomorphism classes of indecomposable finitely generated Gorenstein projecitve modules. Clearly, $A$ is CM-finite if and only if there is a finitely generated module $E$ such that $\Gproj A=\add E$. In this way, $E$ is called to be a \emph{Gorenstein projective generator}. If $\gldim A<\infty$, then $\Gproj A=\proj A$, so $A$ is CM-finite. If $A$ is self-injective, then $\Gproj A=\mod A$, so $A$ is CM-finite if and only if $A$ is representation finite.

Let $A$ be a CM-finite algebra, $E_1,\dots,E_n$ all the pairwise non-isomorphic indecomposable Gorenstein projective $A$-modules. Put $E=\oplus_{i=1}^n E_i$. Then $E$ is a Gorenstein projective generator. We call $\Aus(\Gproj A):=(\End_A E)^{op}$ the \emph{Cohen-Macaulay Auslander algebra}(also called \emph{relative Auslander algebra}) of $A$.

\begin{lemma}[\cite{LZ}]\label{lemma global dimension of Gorenstein Auslander algebra}
Let $A$ be a CM-finite Artin algebra. Then we have the following:

(i) $\gldim \Aus(\Gproj A)=0$ if and only if $A$ is semisimple.

(ii) $\gldim \Aus(\Gproj A)=1$ if and only if $\gldim A=1$.

(iii) $\gldim \Aus(\Gproj A)=2$ if and only if either

(a) $\Gproj A=\proj A$ and $\gldim A=2$, or

(b) $\Gproj A\neq \proj A$ and $A$ is a Gorenstein algebra with $\Gd A\leq 2$.

(iv) If $\Gd A\geq3$, then:
$$\gldim \Aus(\Gproj A)=\Gd A.$$

(v) $A$ is Gorenstein if and only if Cohen-Macaulay Auslander algebra $\Aus(\Gproj A)$ has finite global dimension.
\end{lemma}

Gei{\ss} and Reiten \cite{GR} have shown that gentle algebras are Gorenstein algebras. So their Cohen-Macaulay Auslander algebras have finite global dimensions.

The singularity category of a gentle algebra is characterized by Kalck in \cite{Ka}, we recall it as follows. For a gentle algebra $\Lambda=KQ/\langle I\rangle$, we denote by $\cc(\Lambda)$ the set of equivalence classes (with respect to cyclic permutation) of \emph{repetition-free} cyclic paths $\alpha_1\dots\alpha_n$ in $Q$ such that $\alpha_i\alpha_{i+1}\in I$ for all $i$, where we set $n+1=1$. Moreover, we set $l(c)$ for the \emph{length} of a cycle $c\in\cc(\Lambda)$, i.e. $l(\alpha_1\dots\alpha_n)=n$.

For every arrow $\alpha\in Q_1$, there is at most one cycle $c\in\cc(\Lambda)$ containing it. We define $R(\alpha)$ to be the \emph{left ideal} $\Lambda \alpha$ generated by $\alpha$.
It follows from the definition of gentle algebras that this is a direct summand of the radical $\rad P_{s(\alpha)}$ of the indecomposable projective $\Lambda$-module $P_{s(\alpha)}=\Lambda e_{s(\alpha)}$, where $e_{s(\alpha)}$ is the idempotent corresponding to $s(\alpha)$. In fact, all radical summands of indecomposable projectives arise in this way.

\begin{theorem}[\cite{Ka}]\label{theorem Kalck}
Let $\Lambda=KQ/\langle I\rangle$ be a gentle algebra. Then

\emph{(i)} $\ind \Gproj(\Lambda)=\ind \proj \Lambda \bigcup \{R(\alpha_1),\dots, R(\alpha_n)|c=\alpha_1\cdots \alpha_n\in\cc(\Lambda)\}$.

\emph{(ii)} There is an equivalence of triangulated categories
$$D^b_{sg}(\Lambda)\simeq \prod_{c\in\cc(\Lambda)} \frac{D^b( K)}{[l(c)]} ,$$
where $ D^b(K)/[l(c)]$ denotes the triangulated orbit category, see Keller \cite{Ke}.
\end{theorem}

From Theorem \ref{theorem Kalck}, we get that
$$\ind \Gproj(\Lambda)=\ind \proj \Lambda \bigcup \{R(\alpha_1),\dots, R(\alpha_n)|c=\alpha_1\cdots \alpha_n\in\cc(\Lambda)\}.$$
Furthermore, from its proof, let $c\in\cc(\Lambda)$ be a cycle, which we label as follows: $1\xrightarrow{\alpha_1} 2\xrightarrow{\alpha_2} \cdots \xrightarrow{\alpha_{n-1}}n\xrightarrow{\alpha_n}1$.
Then there are short exact sequences
\begin{equation}\label{equation 1}
0\rightarrow R(\alpha_{i})\xrightarrow{a_i} P_i\xrightarrow{b_i} R(\alpha_{i-1}) \rightarrow0,
\end{equation}
for all $i=1,\dots,n$.

A classification of indecomposable modules over gentle algebras can be deduced from work of Ringel \cite{R} (see e.g. \cite{BR,WW}). For each arrow $\beta$, we denote by $\beta^{-1}$ the formal inverse of $\beta$ with $s(\beta^{-1})=t(\beta)$ and $t(\beta^{-1})=s(\beta)$. A word $w=c_1c_2\cdots c_n$ of arrows and their formal inverse is called a \emph{string} of length $n\geq1$ if $c_{i+1}\neq c_i^{-1}$, $s(c_i)=t(c_{i+1})$ for all $1\leq i\leq n-1$, and no subword nor its inverse is in $I$. We define $(c_1c_2\cdots c_n)^{-1}=c_n^{-1}\cdots c_2^{-1} c_1^{-1}$, and $s(c_1c_2\cdots c_n)=s(c_n)$, $t(c_1c_2\cdots c_n)=t(c_1)$.
We denote the length of $w$ by $l(w)$.
In addition, we also want to have strings of length $0$; be definition, for any vertex $u\in Q_0$, there will be two strings of length $0$, denoted by $1_{(u,1)}$ and $1_{(u,-1)}$, with both
$s(1_{(u,i)})=u=t(1_{(u,i)})$ for $i=-1,1$, and we define $(1_{(u,i)})^{-1}=1_{(u,-i)}$.
We also denote by $\cs(\Lambda)$ the set of all strings over $\Lambda=KQ/\langle I\rangle$.

A \emph{band} $b=\alpha_1\alpha_2\cdots \alpha_{n-1} \alpha_n$ is defined to be a string $b$ with $t(\alpha_1)=s(\alpha_n)$ such that each power $b^m$ is a string, but $b$ itself is not a proper power of any strings. We denote by $\cb(\Lambda)$ the set of all bands over $\Lambda$.

On $\cs(\Lambda)$, we consider the equivalence relation $\rho$ which identifies every string $C$ with its inverse $C^{-1}$. On $\cb(\Lambda)$, we consider the equivalence relation $\rho'$ which identifies every string $C=c_1\dots c_n$ with the cyclically permuted strings $C_{(i)}=c_ic_{i+1}\cdots c_nc_1\cdots c_{i-1}$ and their inverses $C_{(i)}^{-1}$, $1\leq i\leq n$. We choose a complete set $\underline{\cs}(\Lambda)$ of representatives of $\cs(\Lambda)$ relative to $\rho$, and a complete set $\underline{\cb}(\Lambda)$ of representatives of $\cb(\Lambda)$ relative to $\rho'$.

Butler and Ringel showed that each string $w$ defines a unique string module $M(w)$, each band $b$ yields a family of band modules $M(b,m,\phi)$ with $m\geq1$ and $\phi\in \Aut(K^m)$.
Equivalently, one can consider certain quiver morphism $\sigma:S\rightarrow Q$ (for strings) and $\beta:B\rightarrow Q$ (for bands), where $S$ and $B$ are of Dynkin types $\A_n$ and $\tilde{A}_n$, respectively. Then string and band modules are given as pushforwards $\sigma_*(M)$ and $\beta_*(R)$ of indecomposable $KS$-modules $M$ and indecomposable regular $KB$-modules $R$, respectively (see e.g. \cite{WW}).
Let $\underline{\Aut}(K^m)$ be a complete set of representatives of indecomposable automorphisms of $K$-spaces with respect to similarity.

\begin{theorem}[\cite{BR}]
The modules $M(w)$, with $w\in \underline{\cs}(\Lambda)$, and the modules $M(b,m,\phi)$ with $b\in \underline{\cb}(\Lambda)$, with $b\in\underline{\cb}(\Lambda)$, $m\geq1$ and $\phi\in \underline{\Aut}(K^m)$, provide a complete list of indecomposable (and pairwise non-isomorphic) $\Lambda$-modules.
\end{theorem}

In practice, a string $w$ is of form $\alpha_1^{\epsilon_1} \alpha_2^{\epsilon_2}\cdots \alpha_n^{\epsilon_n}$ for $\alpha_i\in Q_1$ and $\epsilon_i=\pm1$ for all $1\leq i\leq n$. So $w$ can be viewed as a walk in $Q$:
\[\xymatrix{ w:\quad 1 \ar@{-}[r]^{\quad\alpha_1} &2 \ar@{-}[r]^{\alpha_2}& \cdots \ar@{-}[r]^{\alpha_{n-1}}& n\ar@{-}[r]^{\alpha_n} &n+1,}\]
where $i\in Q_0$ are vertices of $Q$ and $\alpha_i$ are arrows in either directions. In this way, the equivalence relation $\rho$ induces that
\[\xymatrix{ w:\quad 1 \ar@{-}[r]^{\quad\alpha_1} &2 \ar@{-}[r]^{\alpha_2}& \cdots \ar@{-}[r]^{\alpha_{n-1}}& n\ar@{-}[r]^{\alpha_n} &n+1,}\]
is equivalent to
\[\xymatrix{ w^{-1}:\quad n+1 \ar@{-}[r]^{\quad\quad\alpha_n} &n \ar@{-}[r]^{\alpha_{n-1}}& \cdots \ar@{-}[r]^{\alpha_{2}}& 2\ar@{-}[r]^{\alpha_1} &1.}\]
It is similar to interpret $\rho'$ if $w$ is a band. We denote by $v\sim w$ for two strings $v,w$ if $v$ is equivalent to $w$ under $\rho$.

For any string $w=c_1\dots c_n$, or $w=1_{(u,t)}$, let $u(i)=t(c_{i+1})$, $0\leq i<n$, and $u(n)=s(w)$. Given a vertex $v\in Q_0$, let $I_v=\{ i|u(i)=v\}\subseteq\{0,1,\dots,n\}$.
Denote by $k_v=|I_v|$.
We associate a vector $(k_v )_{v\in Q_0}$ to the string $w$, this vector is denoted by $\dimv w$, and call it the \emph{dimension vector} of $w$.
From \cite{BR}, we get that $\dimv w=\dimv M(w)$.

\subsection{Quiver Grassmannians}
Quiver Grassmannians are varieties parametrizing subrepresentations of a quiver representation. They are introduced by Craw-Boevey(\cite{C-B89}), Schofield(\cite{Sch}) to study the generic properties of quiver representations.

For a finite quiver $Q$, let $A=KQ/I$ be a finite dimensional algebra.
We consider finite dimensional representations $M$ of $A$ over $K$, viewed either as finite dimensional left modules over the path algebra $A$, or as tuples
$$M=((M_i)_{i\in Q_0},(M_\alpha:M_i\rightarrow M_j)_{(\alpha:i\rightarrow j)\in Q_1})$$
consisting of finite dimensional $K$-vector spaces $M_i$ and linear maps $M_\alpha$.

For a representation $M$ of $A$, let $d=\dimv M$.
We consider the affine space
$$R_d(Q)=\bigoplus_{(\alpha:i\rightarrow j)}\Hom_K(M_i,M_j);$$
its points canonically paprametrize representations of $Q$ of dimension vector of $d$. Let $R_d(A)$ be the subvariety of $R_d(Q)$ consisting of representations annihilated by $I$, where its points canonically parametrize representations of $A$ of dimension vector $d$.
The reductive algebraic group $G_d=\prod_{i\in Q_0} \GL(M_i)$ acts naturally on $R_d(Q)$ via base change
$$(g_i)_i\cdot (M_\alpha)_\alpha=(g_jM_\alpha g_i^{-1})_{(\alpha:i\rightarrow j)},$$
such that the orbit $\co_M$ for this action naturally corresponds to the isomorphism classes $[M]$ of representations of $Q$ of dimension vector $d$.
Note that $$\dim G_d-\dim R_d(Q)=\sum_{i\in Q_0} d_i^2-\sum_{\alpha:i\rightarrow j} d_id_j=\langle d,d\rangle_Q.$$ The stabilizer under $G_d$ of a point $M\in R_d(Q)$ is isomorphic to the automorphism group $\Aut_Q(M)$ of the corresponding representation, which is a connected algebraic group of dimension $\dim \End_Q(M)$. In particular, we get
$$\dim \co_M=\dim G_d-\dim \End_Q(M).$$

The constructions and results in the following follow \cite{Sch}, see also \cite{CR,CFR12}.
Additionally to the above, fix another dimension vector $e$ such that $e\leq d$ componentwise, and define the $Q_0$-graded Grassmannian $\Gr_e(d)=\prod_{i\in Q_0}\Gr_{e_i}(M_i)$, which is a projective homogeneous space for $G_d$ of dimension $\sum_{i\in Q_0}e_i(d_i-e_i)$.
We define $\Gr_{e}^Q(d)$, the universal Grassmannian of $e$-dimensional subrepresentations of $d$-dimensional representations of $Q$ as the closed subvariety of $\Gr_e(d)\times R_d(Q)$ consisting of tuples $((U_i\subseteq M_i)_{i\in Q_0},(M_\alpha)_{\alpha\in Q_1})$ such that $M_\alpha(U_i)\subseteq U_j$ for all arrows $(\alpha:i\rightarrow j)\in Q_1$. The group $G_d$ acts on $\Gr_e^Q(d)$ diagonally, such that the projection $p_1:\Gr_e^Q(d)\rightarrow \Gr_e(d)$ and $p_2:\Gr_e^Q(d)\rightarrow R_d(Q)$ are $G_d$-equivariant. In fact, the projection $p_1$ identifies $\Gr_e^Q(d)$ as the total space of a homogeneous bundle over $\Gr_e(d)$ of rank
$$\sum_{(\alpha:i\rightarrow j)\in Q_1}(d_id_j+e_ie_j-e_id_j),$$
and $\Gr_e^Q(d)$ is smooth and irreducible of dimension
$$\dim \Gr_{e}^Q(d)=\langle e,d-e\rangle_Q+\dim R_d(Q).$$
The projection $p_2$ is proper, thus its image is a closed $G_d$-stable subvariety of $R_d$, consisting of representations admitting a subrepresentation of dimension vector $e$.

We define the \emph{quiver Grassmannian} $\Gr^Q_e(M)=p_2^{-1}(M)$ as the fibre of $p_2$ over a point $M\in R_d(Q)$; by definition, it parametrizes $e$-dimensional subrepresentations of the representation $M$.

Similarly, we define the universal quiver Grassmannian $\Gr_e^A(d)$ as the closed subset of $\Gr_e(d)\times R_d(A)$ consisting of pairs $((U_i\subseteq M_i)_{i\in Q_0},(f_\alpha)_{\alpha\in Q_1})$ such that $f_\alpha(U_i)\subseteq U_j$ for all arrows $\alpha:i\rightarrow j$. Using the projection $p_2:\Gr_e^A(d)\rightarrow R_d(A)$, we define the scheme-theoretic quiver Grassmannian $\cg r^A_e(M)=p_2^{-1}(M)$, thus by definition, $\Gr_e^A(M)$ is isomorphic to $\cg r^A_e(M)$ endowed with the reduced structure. However, for any $A$-module $M$, $A$-submodules of $M$ are the same as its $KQ$-submodules, so we also denote $\Gr_e^A(M)$ by $\Gr_e^Q(M)$ or even $\Gr_e(M)$ if there is no confusion.

\begin{theorem}[\cite{CFR13}]\label{lemma smooth of quiver Grassmannian}
Let $Q$ be a quiver, let $M$ be a representation of $KQ$ and let $e$ be a dimension vector for $Q$. Assume that $M$ is a representation for a quotient algebra $A=KQ/I$, such that the following holds: $A$ has global dimension at most $2$, both the injective and the projective dimension of $M$ over $A$ are at most one, and $\Ext^1_A(M,M)=0$. Then the quiver Grassmannian $\Gr_e(M)$ is smooth (and reduced), with irreducible and equidimensional connected components.
\end{theorem}

\section{The functor $\Phi$}

First, we give a characterization of $1$-Gorenstein gentle algebra.

\begin{proposition}\label{proposition structure of 1-gorenstein gentle algebra}
Let $\Lambda=KQ/I$ be a finite dimensional gentle algebra. Then $\Lambda$ is $1$-Gorenstein if and only if for any arrows $\alpha,\beta$ in $Q$ satisfying $0\neq\beta\alpha\in I$, there exists $c\in \cc(\Lambda)$ such that $\alpha,\beta\in c$.
\end{proposition}
\begin{proof}
First, we assume that $\Lambda$ is $1$-Gorenstein. If there are some arrows $\alpha,\beta$ such that $0\neq\beta\alpha\in I$, however, there is no $c\in\cc(\Lambda)$ such that $\alpha,\beta\in c$. Without losing generality, by the definition of $\cc(\Lambda)$, we can assume that there exists one arrow $\alpha_i$ satisfying that there is an arrow $\alpha_{i-1}$ such that $0\neq\alpha_i\alpha_{i-1}\in I$, but there is no arrow $\alpha_{i+1}$ such that $0\neq \alpha_{i+1}\alpha_i\in I$. Locally, the quiver $Q$ looks as the following Figure 1 shows, where the vertices can be coincided.

Note that there exists at most one arrow $\gamma_1$ starts from $i$.
We assume that $j_2,j_3$ are the ending vertices of the longest paths $\beta_s\dots\beta_1$ and $\gamma_1\dots\gamma_t$ starting from $i-1$ and $i$ respectively, with $\beta_1\neq\alpha_i$; $j_1,j_4,j_5$ are the starting vertices of the longest paths $\eta_1\dots \eta_p,\xi_1\dots\xi_q,\delta_1\dots\delta_q$ ending to $i-1$, $i$, $j_3$ respectively, with $\eta_1\neq \alpha_{i-1}$, $\xi_1\neq\alpha_i$ and $\delta_1\neq\gamma_t$.

Denote by $P_{i-1}$ and $P_i$ the indecomposable projective modules corresponding to $i-1$ and $i$ respectively, and $I_{j_3}$ the indecomposable injective module corresponding to $j_3$.
$P_{i-1}$ is a string module with string
$$\xymatrix{j_3 & \cdots\ar[l]_{\gamma_t} & i\ar[l]_{\gamma_1} & i-1\ar[l]_{\alpha_i}  \ar[r]^{\beta_1} &\cdots \ar[r]^{\beta_s}& j_2,}$$
$P_i$ is a string module with string
$$\xymatrix{i \ar[r]^{\gamma_1}&\cdots\ar[r]^{\gamma_t}& j_3. }$$
$I_{j_3}$ is a string module with string
$$\xymatrix{j_1\ar[r]^{\eta_p} &\cdots \ar[r]^{\eta_1} &i-1\ar[r]^{\alpha_i}& i\ar[r]^{\gamma_1} &\cdots \ar[r]^{\gamma_t} &j_3 & \cdots \ar[l]_{\delta_1} &j_5.\ar[l]_{\delta_u}}$$
Then there exists a short exact sequence
$$0\rightarrow P_i\rightarrow I_{j_3} \rightarrow M_1\oplus M_2\rightarrow0,$$
where $M_1,M_2$ are the string modules with their strings
$$\xymatrix{j_1\ar[r]^{\eta_p} &\cdots \ar[r]^{\eta_1} &i-1,}$$
and
$$ \xymatrix{\bullet& \cdots\ar[l]_{\delta_2} &j_5\ar[l]_{\delta_u}}$$
respectively. Note that $0\neq M_1$ is not injective, so $\id P_i>1$, and then $\Lambda$ is not $1$-Gorenstein, a contradiction.

\begin{center}
\setlength{\unitlength}{0.9mm}
\begin{picture}(60,60)
\put(-18,19){\tiny$i-2$}
\put(-10,20){\vector(1,0){10}}
\put(2,19){\tiny$i-1$}
\put(11,20){\vector(1,0){10}}
\put(22,19){$\cdots$}
\put(29,20){\vector(1,0){10}}
\put(40,19){\tiny$j_2$}
\put(6,17){\vector(0,-1){10}}
\put(5,4){\tiny$i$}
\put(7,5){\vector(1,0){10}}
\put(19,4){$\cdots$}
\put(26,5){\vector(1,0){10}}
\put(37,4){\tiny$j_3$}
\put(50,5){\vector(-1,0){10}}
\put(51,5){\circle*{1}}
\put(62,5){\vector(-1,0){10}}
\put(63,4){$\cdots$}
\put(79,5){\vector(-1,0){10}}
\put(80,4){\tiny $j_5$}

\put(2,12){\tiny$\alpha_i$}
\put(-20,2){\tiny$\xi_q$}
\put(-2,2){\tiny$\xi_1$}
\put(10,2){\tiny$\gamma_1$}
\put(30,2){\tiny$\gamma_t$}
\put(45,2){\tiny$\delta_1$}
\put(57,2){\tiny$\delta_2$}
\put(74,2){\tiny$\delta_u$}
\put(14,21){\tiny$\beta_1$}
\put(32,21){\tiny$\beta_s$}
\put(2,27){\tiny$\eta_1$}
\put(2,45){\tiny$\eta_p$}

\put(6,32){\vector(0,-1){10}}
\put(5.5,34){$\vdots$}
\put(6,49){\vector(0,-1){10}}
\put(5,51){\tiny$j_1$}
\put(-6,5){\vector(1,0){10}}
\put(-12,4){$\cdots$}
\put(-23,5){\vector(1,0){10}}
\put(-26,4){\tiny $j_4$}
\put(-8,21){\tiny$\alpha_{i-1}$}
\put(0,-10){Figure 1. The quiver of $Q$.}
\end{picture}
\vspace{1cm}
\end{center}

\begin{center}
\setlength{\unitlength}{0.9mm}
\begin{picture}(60,60)
\put(-18,19){\tiny$i-2$}
\put(-10,20){\vector(1,0){10}}
\put(2,19){\tiny$i-1$}
\put(11,20){\vector(1,0){10}}
\put(22,19){$\cdots$}
\put(29,20){\vector(1,0){10}}
\put(40,19){\tiny$j_2$}
\put(6,17){\vector(0,-1){10}}
\put(5,4){\tiny$i$}
\put(7,5){\vector(1,0){10}}
\put(19,4){$\cdots$}
\put(26,5){\vector(1,0){10}}
\put(37,4){\tiny$j_3$}
\put(50,5){\vector(-1,0){10}}
\put(50,4.5){\tiny $b$}
\put(62,5){\vector(-1,0){10}}
\put(63,4){$\cdots$}
\put(79,5){\vector(-1,0){10}}
\put(80,4){\tiny $j_5$}

\put(53,20){\vector(-1,0){10}}
\put(65,20){\vector(-1,0){10}}
\put(53,19.5){\tiny $a$}
\put(66,19){$\cdots$}
\put(82,20){\vector(-1,0){10}}
\put(83,19){\tiny$j_6$}
\put(47,21){\tiny$\rho_1$}
\put(59,21){\tiny$\rho_2$}
\put(77,21){\tiny$\rho_v$}
\put(2,12){\tiny$\alpha_i$}
\put(-20,2){\tiny$\xi_q$}
\put(-2,2){\tiny$\xi_1$}
\put(10,2){\tiny$\gamma_1$}
\put(30,2){\tiny$\gamma_t$}
\put(45,2){\tiny$\delta_1$}
\put(57,2){\tiny$\delta_2$}
\put(74,2){\tiny$\delta_u$}
\put(14,21){\tiny$\beta_1$}
\put(32,21){\tiny$\beta_s$}
\put(2,27){\tiny$\eta_1$}
\put(2,45){\tiny$\eta_p$}

\put(6,32){\vector(0,-1){10}}
\put(5.5,34){$\vdots$}
\put(6,49){\vector(0,-1){10}}
\put(5,51){\tiny$j_1$}
\put(-6,5){\vector(1,0){10}}
\put(-12,4){$\cdots$}
\put(-23,5){\vector(1,0){10}}
\put(-26,4){\tiny $j_4$}
\put(-8,21){\tiny$\alpha_{i-1}$}
\put(6,3){\vector(0,-1){10}}
\put(7,-2){\tiny$\alpha_{i+1}$}
\put(3,-9){\tiny$i+1$}
\put(0,-18){Figure 2. The quiver of $Q$.}
\end{picture}
\vspace{1.7cm}
\end{center}

Conversely, for any indecomposable projective module $P$, if the corresponding vertex $i-1$ lies at a cycle $c=\alpha_n\dots\alpha_1\in\cc(\Lambda)$, we assume the quiver $Q$ looks as the above Figure 2 shows.

Note that the vertices $j_1,\dots,j_6$ are defined as the above.
Then there exists a short exact sequence
$$0\rightarrow P_{i-1}\rightarrow I_{j_2}\oplus I_{j_3}\rightarrow I_{i-1}\oplus N_1\oplus N_2\rightarrow0,$$
where $N_1$, $N_2$ are the string modules corresponding to strings
$$ \xymatrix{a& \cdots\ar[l]_{\rho_2} &j_6,\ar[l]_{\rho_v}}$$ and
$$ \xymatrix{b& \cdots\ar[l]_{\delta_2} &j_5\ar[l]_{\delta_u}}$$
respectively.

We claim that $N_1$ and $N_2$ are injective modules. If $N_1$ is nonzero and not injective, then there is another arrow $\mu_1$ ending to $a$. Then $\rho_1 \mu_1\in I$, which implies that $\rho$ belongs to a cycle $c'\in\cc(\Lambda)$ by the assumption.
So there exists an arrow $\mu_2$ starting from $j_2$ such that $\mu_2\rho_1\in I$, we get that $\mu_2\beta_s\notin I$, so $\beta_s\dots\beta_1$ is not the longest string starting from $i-1$, a contradiction. Similarly, we can prove that $N_2$ is injective. Therefore, $\id P_{i-1}\leq 1$.

If the indecomposable projective module $P$ satisfies that its corresponding vertex $i-1$ does not lie at any cycle $c\in\cc(\Lambda)$, we assume that the quiver $Q$ looks as the following diagram shows. Note that there are at most one arrow $\alpha_i$ starting from $i-1$, and at most one arrow $\alpha_{i-1}$ ending to $i-1$.
\begin{center}
\setlength{\unitlength}{0.9mm}
\begin{picture}(100,10)
\put(-20,0){\vector(1,0){10}}
\put(-23,-1){\tiny $j_0$}
\put(-9,-1){$\cdots$}
\put(-3,0){\vector(1,0){10}}
\put(8,-1){\tiny $i-2$}
\put(15,0){\vector(1,0){10}}
\put(26,-1){\tiny$i-1$}
\put(33,0){\vector(1,0){10}}
\put(44,-1){\tiny $i$}

\put(46,0){\vector(1,0){10}}
\put(57,-1){$\cdots$}
\put(63,0){\vector(1,0){10}}
\put(74,-1){\tiny $n$}
\put(86,0){\vector(-1,0){10}}
\put(87,-1){\tiny$a$}
\put(99,0){\vector(-1,0){10}}
\put(100,-1){$\cdots$}
\put(116,0){\vector(-1,0){10}}
\put(117,-1){\tiny $j_1$}

\put(-17,1){\tiny $\alpha_1$}
\put(-2,1){\tiny $\alpha_{i-2}$}
\put(16,1){\tiny $\alpha_{i-1}$}
\put(35,1){\tiny $\alpha_i$}
\put(47,1){\tiny $\alpha_{i+1}$}
\put(65,1){\tiny $\alpha_n$}
\put(80,1){\tiny $\beta_1$}
\put(93,1){\tiny $\beta_2$}
\put(110,1){\tiny $\beta_s$}
\put(25,-10){Figure 3. The quiver of $Q$.}
\end{picture}
\vspace{1cm}
\end{center}
Denote by $j_0$ the starting vertex of the longest path $\alpha_{i-1}\dots\alpha_1$, and $n$ the ending vertex of the longest path $\alpha_n\dots\alpha_i$.
Also $j_1$ is the starting vertex of the longest path $\beta_1\dots\beta_s$ ending to $n$.
So $P_{i-1}$ is the string module with string
$$\xymatrix{i-1\ar[r]^{\alpha_i}& i\ar[r]^{\alpha_{i+1}} &\cdots\ar[r]^{\alpha_n}&n.}$$
$I_n$ is the string module with string as Figure 3 shows.
So there is a short exact sequence
$$0\rightarrow P_{i-1}\rightarrow I_n\rightarrow L_1\oplus L_2\rightarrow0,$$
where $L_1$ is the string module with string
$$\xymatrix{j_0\ar[r]^{\alpha_1}&\cdots\ar[r]^{\alpha_{i-2}}&i-2,}$$
and $L_2$ is the string module with string
$$\xymatrix{j_1\ar[r]^{\beta_s}&\cdots \ar[r]^{\beta_2}& a.}$$

We claim that $L_1$ and $L_2$ are injective modules.
If $L_2$ is nonzero and not injective, then there exists one arrow $\gamma_1$ ending to $a$, which implies that $\beta_1\gamma_1\in I$ since $\beta_1\beta_2\notin I$.
By the assumption, we know that there exists one arrow $\gamma_2$ starting from $n$ such that $\gamma_2\beta_1\in I$, then $\gamma_2\alpha_n\notin I$, contradicts to $\alpha_n\dots\alpha_i$ is the longest path starting from $i-1$. So $L_2$ is injective. Similarly, we can prove that $L_1$ is injective. So we get that $\id P_{i-1}\leq1$.

To sum up, for any indecomposable projective module $P$, we get that $\id P\leq1$, which yields that $\id \Lambda\leq1$. Since $\Lambda$ is Gorenstein, we get that $\Lambda$ is $1$-Gorenstein.
\end{proof}

From now on, we only consider gentle algebras $\Lambda$ which is $1$-Gorenstein.

For any gentle algebra $\Lambda$, from \cite{CL2}, we know that $\Lambda$ is CM-finite, and the Cohen-Macaulay Auslander algebra $\Aus(\Gproj \Lambda)=KQ^{Aus}/\langle I^{Aus}\rangle$ is the algebra with the bound quiver $(Q^{Aus},I^{Aus})$ defined as follows:

$\bullet$ the set of vertices $Q^{Aus}_0:=Q_0\bigsqcup Q_1^{cyc}$, where $Q_1^{cyc}=\{\alpha| \alpha\in\cc(\Lambda)\}$;

$\bullet$ the set of arrows $Q^{Aus}_1:=Q_1^{ncyc}\bigsqcup (Q_1^{cyc})^{\pm}$, where $Q_1^{ncyc}=Q_1\setminus Q_1^{cyc}$, i.e. arrows do not appear in any cyclic paths in $\cc(\Lambda)$, $(Q_1^{cyc})^{+}=\{\alpha^+:s(\alpha)\rightarrow \alpha | \alpha\in Q_1^{cyc} \}$ and
$(Q_1^{cyc})^{-}=\{\alpha^-:\alpha\rightarrow t(\alpha) | \alpha\in Q_1^{cyc} \}$.

\noindent The ideal $I^{Aus}:= \{\beta^+\alpha^-| \beta\alpha \in I,\alpha,\beta\in Q_1^{cyc}\}\bigcup\{\beta\alpha|\beta\alpha\in I,\alpha,\beta\in Q_1^{ncyc}\}$.

Before going on, let us fix some notations.
Let $\Lambda$ be a gentle algebra and $\Gamma$ be its Cohen-Macaulay Auslander algebra.

For any $M=((M_i)_{i\in Q_0},(M_\alpha:M_i\rightarrow M_j)_{(\alpha:i\rightarrow j)\in Q_1})\in\mod \Lambda$,
Define a $\Gamma$-module $\widehat{M}=((N_i,N_\alpha)_{i\in Q_0,\alpha\in Q_1^{cyc}},(N_\beta)_{\beta \in Q_1^{Aus}})$ as follows:

$\bullet$ For any $i\in Q_0\subseteq Q_0^{Aus}$, we set $N_i=M_i$; For any $\alpha\in Q_1^{cyc}\subseteq Q_1^{Aus}$, we set $N_\alpha=\Im M_\alpha$.

$\bullet$ For any arrow in $Q_1^{Aus}$, if it is of form $(\beta:i\rightarrow j)\in Q_1^{ncyc}$, then we set $N_\beta=M_\beta$; if it is of form $\beta^+:i\rightarrow \beta$, or of form
$\beta^-:\beta\rightarrow j$ for some $(\beta:i\rightarrow j)\in Q_1^{cyc}$, we set $N_{\beta^+}$ and $N_{\beta^-}$ to be the natural morphisms $(N_i=M_i)\rightarrow (\Im M_\beta=N_{\beta})$ and $(N_\beta=\Im M_\beta)\rightarrow (M_j=N_j)$ respectively induced by $M_\beta:M_i\rightarrow M_j$.

It is easy to see that $\widehat{M}$ is actually a $\Gamma$-module. Since $\Im$ is a functor, we can define a functor $\Phi:\mod \Lambda\rightarrow \mod \Gamma$ such that $\Phi(M):=\widehat{M}$, with the natural definition on morphisms.

Since $(Q,I)$ is a subquiver of $(Q^{Aus},I^{Aus})$, i.e. $\Lambda$ is a subalgebra of $\Gamma$, we get a restriction functor $\res:\mod \Gamma\rightarrow \mod\Lambda$.
Explicitly, for any $N=((N_i,N_\alpha)_{i\in Q_0,\alpha\in Q_1^{cyc}},(N_\beta)_{\beta \in Q_1^{Aus}})\in\mod \Gamma$, $\res(N)$ is defined as follows:

$\bullet$ For any $i\in Q_0$, $(\res(N))_i=N_i$;

$\bullet$ For any arrow $(\alpha:i\rightarrow j)\in Q_1$, if $\alpha\in Q_1^{ncyc}$, we set $(\res N)_\alpha=N_\alpha$; if $\alpha\in Q_1^{cyc}$, we set $(\res(N))_\alpha=N_{\alpha^-}N_{\alpha^+}$.

\begin{lemma}
We have $\res(\Phi(M))\simeq M$ naturally.
\end{lemma}
\begin{proof}
This follows from the definition of the functors $\Phi$ and $\res$ immediately.
\end{proof}
Similar to \cite[Lemma 5.3]{CFR13}, we also have the following weak adjunction properties:

\begin{lemma}\label{lemma property of phi}
(i) For all $M\in\mod \Lambda$ and $N\in\mod \Gamma$, the natural maps
$$\Hom_\Gamma(\widehat{M},N)\xrightarrow{\res} \Hom_\Lambda(\res \widehat{M},\res N)\simeq \Hom_\Lambda(M,\res N)$$
and
$$\Hom_\Gamma(N,\widehat{M})\xrightarrow{\res}\Hom_\Lambda(\res N,\res\widehat{M})\simeq \Hom_{\Lambda}(\res N,M)$$
are injective.

(ii) The functor $\Phi$ is fully faithful. In particular, $\Phi$ preserves injective and surjective morphisms.
\end{lemma}
\begin{proof}
(i) For any morphism $f:\widehat{M}\rightarrow N$, then it is of form $(f_i)_{i\in Q_0}\bigcup(f_\alpha)_{\alpha\in Q_1^{cyc}}$. If $\res(f)=0$, then $\res(f)=(f_i)_{i\in Q_0}=0$. Let $M=((M_i)_{i\in Q_0},(M_\alpha)_{(\alpha:i\rightarrow j)\in Q_1})$ and $N=((N_i)_{i\in Q_0}\cup(N_\alpha)_{\alpha\in Q_1},(N_\beta)_{\beta\in Q_1^{Aus}})$.
For any $(\alpha:i\rightarrow j)\in Q_1^{cyc}$, since $f$ is a morphism, we get the following commutative diagram
$$\xymatrix{ M_i\ar[r] \ar[d]^{f_i}& \Im M_\alpha \ar[r]\ar[d]^{f_\alpha}& M_j \ar[d]^{f_j}\\
N_i\ar[r]^{N_{\alpha^+}} &N_\alpha \ar[r]^{N_{\alpha^-}} &N_j.}$$
Since $f_i=f_j=0$ and the first map in the upper row is surjective, we get that $f_\alpha=0$.
So $\Hom_\Gamma(\widehat{M},N)\xrightarrow{\res} \Hom_\Lambda(\res \widehat{M},\res N)$ is injective. The second statement can be proved dually.

(ii) For all $M,N$ in $\mod \Lambda$, we have a chain of maps
$$\Hom_{\Lambda}(M,N)\xrightarrow{\Phi}\Hom_{\Gamma}(\widehat{M},\widehat{N})\xrightarrow{\res}\Hom_\Lambda(\res\widehat{M},\res\widehat{N})\simeq \Hom_\Lambda(M,N),$$
whose composition is the identity, thus the first map is injective. The second map is injective due to (i). So $\Phi$ is fully faithful.
\end{proof}

Since $\Lambda$ and $\Gamma$ are gentle algebras, their indecomposable modules are either string modules or band modules. We describe the action of $\Phi$ and $\res$ on string modules as follows.

$\bullet$ For a string $w=\alpha_1^{\epsilon_1}\alpha_2^{\epsilon_2}\dots\alpha_n^{\epsilon_n} \in\cs(\Lambda)$, denote the string module by $M(w)$. For $i=1,\dots,n$, if $\alpha_i\in Q_1^{cyc}$, we replace $\alpha_i$ by $\alpha_i^-\alpha_i^+$, and get a word in $\Gamma$, denote it by $\iota(w)$. Then it is easy to see that $\iota(w)\in \cs(\Gamma)$, we denote its string module by $N(\iota(w))$. Note that
$$\dimv N(\iota(w))=\dimv M(w)+\sum_{\alpha_i\in Q_1^{cyc}} \dimv S_{\alpha_i},$$
where $S_{\alpha_i}$ is the simple module corresponding to $\alpha_i\in Q_1^{cyc}\subseteq Q_0^{Aus}$.
In this way, we get a map $\iota:\cs(\Lambda)\rightarrow \cs(\Gamma)$, which is injective. It is easy to see that $\Phi(M(w))=N(\iota(w))$.

$\bullet$ For a string $v=\beta_1\beta_2\dots\beta_n \in\cs(\Gamma)$, denote the string module by $N(v)$. Denote by $v'$ the longest substring of $v$ such that $s(v'),t(v')\in Q_0\subseteq Q_0^{Aus}$. Note that $l(v)-l(v')\leq2$ and $l(v'')-l(v)\leq2$. If $\alpha^-\alpha^+$ (or its inverse) appears as a subword of $v'$ for any arrow $\alpha\in Q_1^{cyc}$, we replace $\alpha^-\alpha^+$ (or its inverse) by $\alpha$ (or $\alpha^{-1}$), after doing this repeatedly, finally we can get a word in $\Lambda$, denote it by $\pi^-(v)$. Then it is easy to see that $\pi^-(v)\in \cs(\Lambda)$, we denote its string module by $M(\pi^-(v))$. Note that if $\dim N(v)=(k_i)_{i\in Q_0}+(k_\alpha)_{\alpha\in Q_1^{cyc}}$, then $\dim M(\pi^-(v))=(k_i)_{i\in Q_0}$.
In this way, we get a surjective map $\pi^-:\cs(\Gamma)\rightarrow \cs(\Lambda)$, in fact, $\pi^-\iota=\Id$. Easily, $\res(N(v))=M(\pi^-(v))$.

\begin{lemma}\label{lemma Phi preserves projectives and injectives}
Let $\Lambda$ be a gentle algebra, and $\Gamma$ be its Cohen-Macaulay Auslander algebra. Then for any projective $\Lambda$-module $P$, and any injective $\Lambda$-module $I$, we have
$\Phi(P)$ is a projective $\Gamma$-module, and $\Phi(I)$ is an injective $\Gamma$-module. In other words, $\Phi$ preserves projectives and injectives.
\end{lemma}
\begin{proof}
From \cite{Ka}, we get that the indecomposable projective $\Lambda$-modules $P$ are of the following form:

\begin{center}
\setlength{\unitlength}{0.9mm}
\begin{picture}(100,50)
\put(0,50){\circle*{1}}
\put(0,50){\vector(0,-1){9}}
\put(0,40){\circle*{1}}
\put(0,40){\vector(0,-1){9}}
\put(0,30){\circle*{1}}
\put(-0.5,23){$\vdots$}
\put(0,20){\circle*{1}}
\put(0,20){\vector(0,-1){9}}
\put(0,10){\circle*{1}}

\put(-4,45){\tiny$\alpha_1$}
\put(-4,35){\tiny$\alpha_2$}
\put(-4,15){\tiny$\alpha_l$}

\put(20,30){$\mbox{or}$}

\put(20,10){\circle*{1}}

\put(30,20){\vector(-1,-1){9}}
\put(30,20){\circle*{1}}

\put(31,21){$\cdot$}
\put(33,23){$\cdot$}
\put(32,22){$\cdot$}
\put(36,26){\circle*{1}}
\put(46,36){\circle*{1}}
\put(46,36){\vector(-1,-1){9}}
\put(56,46){\circle*{1}}
\put(56,46){\vector(-1,-1){9}}

\put(56,46){\vector(1,-1){9}}
\put(66,36){\circle*{1}}
\put(67,33){$\cdot$}
\put(68,32){$\cdot$}
\put(69,31){$\cdot$}
\put(72,30){\circle*{1}}
\put(72,30){\vector(1,-1){9}}
\put(82,20){\circle*{1}}

\put(21,16){\tiny$\beta_m$}
\put(38,33){\tiny$\beta_2$}
\put(48,43){\tiny$\beta_1$}
\put(60,43){\tiny$\gamma_1$}
\put(77,26){\tiny$\gamma_n$}

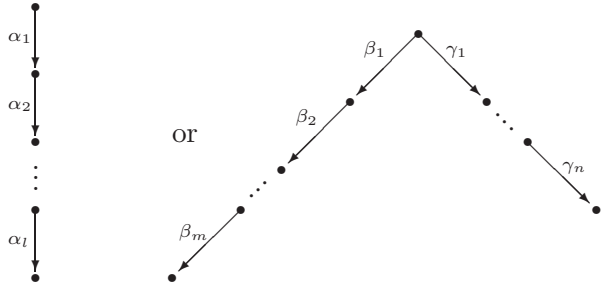
\put(0,0){Figure 3. The strings of projective modules.}
\end{picture}
\end{center}

They correspond to the strings $\alpha_l\dots \alpha_1$ and $\beta_m\dots \beta_1\gamma_1^{-1}\dots \gamma_n^{-1}$, respectively. In particular, $\alpha_l\dots \alpha_1$, $\beta_m\dots\beta_1$ and $\gamma_n\dots \gamma_1$ are maximal, e.g. there does not exist $\alpha\in Q_1$ such that $\alpha\alpha_l\notin I$.

In the first case, $\iota(\alpha_l\dots\alpha_1)\in\cs(\Lambda)$ is also maximal. In fact, if there exists an arrow in $Q_1^{Aus}$, which is either of form $\xi\in Q_1\subseteq Q_1^{Aus}$ or of form $\xi^{+}$ for some $\xi\in Q_1$, such that $\xi\iota(\alpha_l\dots \alpha_1)$ or $\xi^+\iota(\alpha_l\dots\alpha_1)$ is also a string, then
$\xi\alpha_l\notin I$ in both cases, so $\xi\alpha_l\dots \alpha_1$ is a string in $\cs(\Lambda)$, a contradiction.

Note that we can view $s(\alpha_1)$ to be a vertex in $Q_0^{Aus}$. From the structure of $Q^{Aus}$, there is no other arrow than $\alpha_1$ (if $\alpha_1\in Q_1^{ncyc}$ ) or $\alpha_1^+$ (if $\alpha_1\in Q_1^{cyc}$) starting from $s(\alpha_1)$. So $\iota(\alpha_l\dots\alpha_1)$ is the string of the indecomposable projective $\Gamma$-module corresponding to the vertex $s(\alpha_1)$ for $\Gamma$ is gentle.

The second case is similar to the first one, we omit the proof here.

For $\Phi$ preserves injectives, it is dual to the above, we omit the proof here.
\end{proof}

\begin{lemma}\label{lemma projective injective dimension of modules}
Let $\Lambda$ be a $1$-Gorenstein gentle algebra, and $\Gamma$ be its Cohen-Macaulay Auslander algebra. Then for any $M\in\Gproj \Lambda$, $\pd \Phi(M)\leq1$ and $\id \Phi(M)\leq 1$.
\end{lemma}
\begin{proof}
We only prove it for the case when $M$ is an indecomposable Gorenstein projective module.

{\bf Case (1).} If $M$ is projective, then Lemma \ref{lemma Phi preserves projectives and injectives} yields that $\Phi(M)$ is also projective. We denote by $i-1$ the vertex corresponding to the indecomposable projective module $M$.

{\bf Case (1a).} If the corresponding vertex $i-1$ lies at a cycle $c=\alpha_n\dots\alpha_1\in\cc(\Lambda)$, we assume the quiver $Q$ looks as Figure $2$ shows.
Then there is short exact sequence:
$$0\rightarrow P_{i-1}\xrightarrow{f} I_{j_2}\oplus I_{j_3}\xrightarrow{g} I_{i-1}\oplus I_a\oplus I_b\rightarrow0,$$
Lemma \ref{lemma property of phi} implies $\Phi(f)$ is also injective.
We denote by $U_1$ (resp. $U_2$) the indecomposable $\Gamma$-module corresponding to the string $\iota(\rho_2\dots\rho_v)$ (resp. $\iota(\delta_2\dots\delta_u) $)if $\rho_1\in Q_1^{ncyc}$ (resp. $\delta_1\in Q_1^{ncyc}$), and to the string $\rho_1^+\iota(\rho_2\dots\rho_v)$ (resp. $\delta_1^+\iota(\delta_2\dots\delta_u)$) if $\rho_1\in Q_1^{cyc}$ (resp. $\delta\in Q_1^{cyc}$).

We claim that $U_1$ and $U_2$ are injective. We only prove it for $U_1$. If $\rho_1\in Q_1^{ncyc}$, then it follows from Lemma \ref{lemma Phi preserves projectives and injectives} that $U_1$ is injective. Otherwise, if $\rho_1\in Q_1^{cyc}$, then $\rho_1^+\iota(\rho_2\dots\rho_v)$ is a string with ending vertex $\rho_1\in Q_1^{Aus}$, and $\iota(\rho_2\dots\rho_v)$ is the string of an indecomposable injective module, which implies $\iota(\rho_2\dots\rho_v)$ is maximal. Note that there is only one arrow $\rho_1^+$ with ending vertex $\rho_1$. So $\rho_1^+\iota(\rho_2\dots\rho_v)$ is the string of the indecomposable injective module corresponding to vertex $\rho_1\in Q_1^{Aus}$.

By the above construction, it is easy to see that the following sequence is an injective resolution of $\Phi(M)=\Phi(P_{i-1})$:
$$0\rightarrow \Phi(P_{i-1})\xrightarrow{\Phi(f)} \Phi(I_{j_2})\oplus \Phi(I_{j_3})\rightarrow \Phi(I_{i-1})\oplus U_1\oplus U_2\rightarrow0,$$
since $\Phi$ preserves injectives. So $\id \Phi(P_{i-1})\leq1$.

{\bf Case (1b).} If the corresponding vertex $i-1$ does not lie at any cycle in $\cc(\Lambda)$, we assume that the quiver $Q$ looks as Figure $3$ shows. The remaining is similar to Case (1a), we omit the proof here.

{\bf Case (2).} If $M=R(\alpha_{i-1})$ for some $\alpha_{i-1}$ lying on some cycle $c=\alpha_n\dots\alpha_1\in\cc(\Lambda)$,
we assume that the quiver $Q$ looks as Figure $2$ shows. Then $R(\alpha_{i-1})$ is the string module with its string $\beta_s\dots \beta_1$.
There is a short exact sequence
$$0\rightarrow R(\alpha_i)\xrightarrow{f} P_{i-1}\xrightarrow{g} R(\alpha_{i-1})\rightarrow0.$$
So $\Phi(R(\alpha_{i-1}))$ is the $\Gamma$-module corresponding to the string $\iota(\beta_s\dots \beta_1)$.
Denote by $U_{\alpha_{i}}$ the $\Gamma$-module corresponding to the string $\iota(\gamma_t\dots \gamma_1)\alpha_i^-\in\cc(\Gamma)$.
Since there is only one arrow $\alpha_i^-$ starting from $\alpha_i\in Q_0^{Aus}$, and $\gamma_t\dots \gamma_1$ is maximal, we get that $U_{\alpha_{i}}$ is the indecomposable projective $\Gamma$-module corresponding to the vertex $\alpha_i$.
From this construction, it is easy to see that the following sequence is a projective resolution of $\Phi(R(\alpha_{i-1}))$:
$$0\rightarrow U_{\alpha_{i}}\rightarrow \Phi(P_{i-1})\xrightarrow{\Phi(g)} \Phi(R(\alpha_{i-1}))\rightarrow0,$$
since $\Phi$ preserves projectives and epimorphisms. So $\pd \Phi(R(\alpha_{i-1}))\leq1$.

We denote by $J_{\alpha_{i-1}}$ the indecomposable injective $\Gamma$-module corresponding to the vertex $\alpha_{i-1}\in Q_1^{Aus}$, also $U_1$ the indecomposable $\Gamma$-module corresponding to the string $\iota(\rho_2\dots\rho_v)$ if $\rho_1\in Q_1^{ncyc}$, and to the string $\rho_1^+\iota(\rho_2\dots\rho_v)$ if $\rho_1\in Q_1^{cyc}$. By the above, $U_1$ is also injective.
Then the following sequence is an injective resolution of $\Phi(R(\alpha_{i-1}))$:
$$0\rightarrow \Phi(R(\alpha_{i-1}))\rightarrow \Phi(I_{j_2})\rightarrow J_{\alpha_{i-1}}\oplus U_1\rightarrow0.$$
So $\id\Phi(R(\alpha_{i-1}))\leq1$.

In conclusion, for any $M\in\Gproj \Lambda$, we have that $\pd \Phi(M)\leq1$ and $\id \Phi(M)\leq 1$.
\end{proof}

\begin{lemma}\label{lemma rigid of modules}
Let $\Lambda$ be a $1$-Gorenstein gentle algebra, and $\Gamma$ be its Cohen-Macaulay Auslander algebra. Then for any $M\in\Gproj \Lambda$, $\Ext_\Gamma^1(\Phi(M),\Phi(M))=0$.
\end{lemma}
\begin{proof}
We only need prove that for any two indecomposable Gorenstein projective $\Lambda$-modules $M_1,M_2$, $\Ext_\Gamma^1(\Phi(M_1),\Phi(M_2))=0$.

If $M_1$ is projective, then $\Phi(M_1)$ is also projective, so $\Ext_\Gamma^1(\Phi(M_1),\Phi(M_2))=0$ for any $M_2\in\Gproj \Lambda$. For this, we assume that $M_1=R(\alpha_{i-1})$ for some $\alpha_{i-1}\in Q_1^{cyc}$.
We also assume the quiver $Q$ looks as Figure $2$ shows.
Then there is a short exact sequence:
\begin{equation}\label{equation 2}
0\rightarrow U_{\alpha_{i}}\xrightarrow{f} \Phi(P_{i-1})\xrightarrow{\Phi(g)} \Phi(R(\alpha_{i-1}))\rightarrow0,
\end{equation}
where $U_{\alpha_{i}}$ is the $\Gamma$-module corresponding to the string $\iota(\gamma_t\dots \gamma_1)\alpha_i^-\in\cs(\Gamma)$, which is the indecomposable projective module corresponding to the vertex $\alpha_{i}\in Q_0^{Aus}$;
$\Phi(P_{i-1})$ is the indecomposable projective $\Gamma$-module corresponding to the vertex $i-1\in Q_0^{Aus}$.
Note that $f$ is induced by $(\alpha_i^+:i-1\rightarrow \alpha_i)\in Q_1^{Aus}$.

If $M_2$ is projective, which is of form $P_k$ corresponding to the vertex $k$, Then $\Phi(M_2)=\Phi(P_k)$ is also projective.
For any nonzero morphism $h:U_{\alpha_{i}}\rightarrow\Phi(P_k)$, we can assume that $h$ is induced by a nonzero path $l$ from $k$ to $\alpha_{i}$ in $\Gamma$. Since $k\in Q_0\subseteq Q_0^{Aus}$ and $k\neq\alpha_i$, the length of $l$ is not zero.
There is only one arrow $\alpha_i^+$ ending to $\alpha_i\in Q_1^{Aus}$, so $l=\alpha_i^+l'$ for some nonzero path $l'$, which implies that $h$ factors through $f$. So we get that
$$\Hom_\Gamma(\Phi(P_{i-1}), \Phi(P_k))\xrightarrow{\Hom_\Gamma(f,\Phi(P_k))}\Hom_\Gamma(U_{\alpha_i}, \Phi(P_k))$$
is surjective. By applying $\Hom_\Gamma(-,\Phi(P_k))$ to Sequence (\ref{equation 2}), we get that
$$\Ext^1_\Gamma(\Phi(R(\alpha_{i-1})),\Phi(P_k) )=0.$$

If $M_2=R(\beta_{j-1})$, where $\beta_{j-1}$ lies on a cycle $\beta_m\dots\beta_1\in\cc(\Lambda)$,
we denote by $\sigma_k\dots \sigma_1$ be the string of $R(\beta_{j-1})$:
\[\xymatrix{ j-1\ar[r]^{\sigma_1} & \bullet \ar[r]^{\sigma_2}& \cdots \ar[r]^{\sigma_k}&\bullet.}\]

For any nonzero morphism $p:U_{\alpha_{i}}\rightarrow \Phi(R(\beta_{j-1}))$, since $U_{\alpha_i}$ is the indecomposable projective module corresponding to the vertex $\alpha_i$, we assume that $p$ is induced by a nonzero path $l$ from $j-1$ to $\alpha_i$.
Since there is only one arrow $\alpha_i^+$ ending to $\alpha_i$ and $j-1\in Q_0\subseteq Q_0^{Aus}$, we get that $l=l'\alpha_i^+$.
Note that $f$ is induced by $\alpha_i^+$. We get that $p$ factors through $f$ as $p=p'f$ for some morphism $p':\Phi(P_{i-1})\rightarrow \Phi(R(\beta_{j-1}))$.
Therefore, $$\Hom_\Gamma(\Phi(P_{i-1}), \Phi(R(\beta_{j-1})))\xrightarrow{\Hom_\Gamma(f,\Phi(R(\beta_{j-1})))}\Hom_\Gamma(U_{\alpha_i}, \Phi(R(\beta_{j-1})))$$
is surjective. By applying $\Hom_\Gamma(-,\Phi(P_k))$ to Sequence (\ref{equation 2}), we get that
$$\Ext^1_\Gamma(\Phi(R(\alpha_{i-1})),\Phi(R(\beta_{j-1})))=0.$$

In conclusion, we get that $\Ext_\Gamma^1(\Phi(M_1),\Phi(M_2))=0$ for any two indecomposable Gorenstein projective $\Lambda$-modules $M_1,M_2$. So for any $M\in\Gproj \Lambda$, $\Ext_\Gamma^1(\Phi(M),\Phi(M))=0$.
\end{proof}

\begin{theorem}\label{theorem smoothness of quiver grassmannians}
Let $\Lambda=KQ/\langle I\rangle$ be a gentle algebra which is $1$-Gorenstein, and $\Gamma=KQ^{Aus}/\langle I^{Aus}\rangle$ its Cohen-Macaulay Auslander algebra.
Let $M$ be a Gorenstein projective $\Lambda$-module and let $e$ be a dimension vector for $Q^{Aus}$. Then

(i) $\Gamma$ has global dimension at most two, both the injective and the projective dimension of $\Phi(M)$ over $\Gamma$ are at most one, and $\Ext_\Gamma^1(\Phi(M),\Phi(M))=0$;

(ii) The quiver Grassmannian $\Gr_e(\Phi(M))$ is smooth (and reduced), with irreducible and equidimensional connected components.
\end{theorem}
\begin{proof}
It follows from Lemma \ref{lemma projective injective dimension of modules}, Lemma \ref{lemma rigid of modules} and Lemma \ref{lemma smooth of quiver Grassmannian} immediately.
\end{proof}

\section{Construction of the desingularizations}
This section is based on \cite[Section 7]{CFR13}. In this section, we always assume that $\Lambda=KQ/\langle I\rangle$ to be a $1$-Gorenstein gentle algebra, and $\Gamma$ be its Cohen-Macaulay Auslander algebra. For a given Gorenstein projective $\Lambda$-module $M$, a dimension vector $e$ and an isomorphism classes $[N]$ of $\Lambda$-modules, it is proved in \cite[Section 2.3]{CFR12} that the subset $\cs_{[N]}$ of $\Gr_e(M)$, consisting of the submodules which are isomorphic to $N$, is locally closed and irreducible of dimension $\dim \Hom_\Lambda(N,M)-\dim \End_\Lambda(N)$.

\begin{lemma}[see e.g. \cite{CGL}]
If $\Lambda$ is $1$-Gorenstein, then the Gorenstein projective modules are just torsionless modules. In particular, for any submodule $N$ of any Gorenstein projective $\Lambda$-module $M$, $N$ is also Gorenstein projective.
\end{lemma}

\begin{lemma}\label{lemma dimension vector of phi(M)}
For any Gorenstein projective $\Lambda=KQ/\langle I\rangle$-module $N=((N_i)_{i\in Q_0},(N_\alpha)_{\alpha \in Q_1})$, the underlying space of $\widehat{N}$ is
$$\bigoplus_{i\in Q_0}\widehat{N}_i\oplus \bigoplus_{\alpha\in Q_1^{cyc}} \widehat{N}_\alpha.$$ Then

(i) $\dim \widehat{N}_i=\dim N_i=\dim\Hom_{\Lambda}(P_i,N)$, where $P_i$ is the indecomposable projective module corresponding to $i$.

(ii) $\dim \widehat{N}_\alpha=\dim\Hom_{\Lambda}(P_{s(\alpha)},N)-\dim \Hom_\Lambda(R(\beta),N)$, where $P_{s(\alpha)}$ is the indecomposable projective module corresponding to $s(\alpha)$, and $\beta$ is the arrow satisfying $\alpha\beta\in I$.
\end{lemma}
\begin{proof}
Without losing generality, we assume that $N$ is indecomposable.

(i) is obvious from the definition of $\Phi$. For (ii), by Sequence (\ref{equation 1}), there is an exact sequence
$$0\rightarrow R(\alpha)\xrightarrow{f} P_{s(\alpha)}\xrightarrow{g} R(\beta)\rightarrow0.$$
Note that $N$ is a string module. Let $w$ be the string of $N$. By the definition of $\Phi$ and $\iota$, it is easy to see that $\dim\widehat{N}_\alpha$ is equal to the times of $w$ passing through $\alpha$.
Let $v$ be the string of $R(\beta)$. Then the quiver of $Q$ locally looks like as the following.
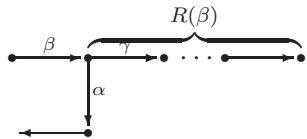
\begin{center}
\setlength{\unitlength}{1mm}
\begin{picture}(60,20)
\put(0,15){\circle*{1}}
\put(0,15){\vector(1,0){9}}
\put(10,15){\circle*{1}}
\put(10,15){\vector(1,0){9}}

\put(10,15){\vector(0,-1){9}}
\put(10,5){\circle*{1}}
\put(10,5){\vector(-1,0){9}}
\put(10,15){\vector(1,0){9}}
\put(20,15){\circle*{1}}
\put(22,14){$\cdots$}
\put(28,15){\circle*{1}}
\put(28,15){\vector(1,0){9}}
\put(38,15){\circle*{1}}
\put(4,16){\tiny$\beta$}
\put(14,16){\tiny$\gamma$}
\put(10.5,10){\tiny$\alpha$}
\put(10,16){$\overbrace{\hspace{2.8cm}}^{R(\beta)}$}
\end{picture}
\end{center}
Note that $v$ is maximal and $\Hom_{\Lambda}(P_{s(\alpha)},N)$ is equal to the times of $w$ passing through $s(\alpha)$.
We get that $\dim\Hom_\Lambda(R(\beta),N)$ is equal to the times of $w$ passing through $\gamma$ if the length of $v$ is not zero, or $s(\alpha)$ is an ending point of $v$ if the length of $v$ is zero. So the times of $w$ passing through $\alpha$ is equal to $\dim \Hom_\Lambda(P_{s(\alpha)},N)-\dim \Hom_\Lambda(R(\beta),N)$, which is also equal to $\dim \widehat{N}_\alpha$ for any $\alpha\in Q_1^{cyc}$.
\end{proof}

\begin{proposition}
Suppose that $\cs_{[U]}$ has non-empty intersection with $\overline{\cs_{[N]}}$. Then $\dimv \widehat{U}\leq \dimv \widehat{N}$ componentwise as dimension vectors of representations of $\Gamma$.
\end{proposition}
\begin{proof}
Following \cite[Section 2.3]{CFR12}, we write
$$\Hom^0(e,M):=\{(N,f)|N\in \mod \Lambda,\dimv N=e\mbox{ and } f:N\rightarrow M\mbox{ is injective}\}.$$
Then $\Gr_e(M)$ is isomorphic to the quotient $\Hom^0(e,M)/G_e$. We have a canonical map
$$p:\Hom^0(e,M)\rightarrow R_e(\Lambda)$$
which maps $(N,f)$ to $N$.
We also denote by $\pi: \Hom^0(e,M)\rightarrow \Gr_e(M)$ the locally trivial $G_e$-principle bundle. Then the stratum $\cs_{[N]}$ is then defined by $\pi(p^{-1}(\co_{[N]}))$.

Similar to \cite[Proposition 7.2]{CFR13}, we know that if $\cs_{[U]}$ has non-empty intersection with $\overline{\cs_{[N]}}$, then $\co_{[U]}\subset \overline{\co_{[N]}}$. So $\dim \Hom_\Lambda(P,U)=\dim \Hom_\Lambda(P,N)$ for all projective representations $P$, and $\dim\Hom_\Lambda(X,U)\geq \dim\Hom_\Lambda(X,N)$ for all non-projectives $X$ (see \cite{Bo}).
Note that $\dim U_i=\dim N_i$ for any $i\in Q_0$.
Now consider the dimension vector of $\widehat{U}$, resp. of $\widehat{N}$.
Then Lemma \ref{lemma dimension vector of phi(M)} (i) implies that $\dim \widehat{U}_i=\dim U_i=\dim N_i=\dim \widehat{N}_i$ for any $i\in Q_0\subseteq Q_0^{Aus}$.
For any $\alpha\in Q_1^{cyc}\subset Q_1^{Aus}$, there exists an arrow $\beta$ such that $\alpha\beta\in I$. Lemma \ref{lemma dimension vector of phi(M)} (ii) yields that
\begin{eqnarray*}
\dim \widehat{U}_\alpha&=&\dim\Hom_{\Lambda}(P_{s(\alpha)},U)-\dim \Hom_\Lambda(R(\beta),U)\\
&\leq& \dim \Hom_{\Lambda}(P_{s(\alpha)},N)-\dim \Hom_\Lambda(R(\beta),N),
\end{eqnarray*}
since $R(\beta)$ is not projective.
This proves $\dimv \widehat{U}\leq \dimv\widehat{N}$ componentwise.
\end{proof}

\begin{definition}[\cite{CFR13}]
We call $[N]$ a generic subrepresentation type of $M$ of dimension vector $e$ if the stratum $\cs_{[N]}$ of $\Gr_e(M)$ is open. Denote by $\gsub_e(M)$ the set of all generic subrepresentation types.
\end{definition}

In case $[N]\in\gsub_e(M)$, the closure $\overline{\cs_{[N]}}$ is an irreducible component of $\Gr_e(M)$, and every irreducible component arises in this way.

For any two Gorenstein projective $\Lambda$-modules $M$ and $N$, we consider quiver Grassmannians for $\Gamma$ of form $\Gr_{\dimv \widehat{N}}(\widehat{M})$. For $[N]\in\gsub_e(M)$, we consider the map
$$\pi_{[N]}:\Gr_{\dimv \widehat{N}}(\widehat{M})\rightarrow \Gr_e(M)$$
given by $(F\subset\widehat{M})\mapsto (\res(F)\subset M)$.

\begin{proposition}\label{proposition fibre of pi}
For $M,e$ and $[N]$ as above and a point $(U\subset M)$ in $\Gr_e(M)$, we have an isomorphism
$$\pi_{[N]}^{-1}(U\subset M)\cong \Gr_{\dimv\widehat{N}-\dimv\widehat{U}}(\widehat{M}/\widehat{U}).$$
\end{proposition}
\begin{proof}
More precisely, we prove that
$$\pi_{[N]}^{-1}(U\subset M)\cong \{F\subset \widehat{M}| \dimv F=\dimv \widehat{N},\widehat{U}\subset F\}.$$

By definition of the map $\pi_{[N]}$, this immediately reduces to the following statement:

\noindent Suppose we are given a subrepresentation $U\subset M$ of dimension vector $e$ and a subobject $F\subset\widehat{M}$ such that $\dimv F=\dimv \widehat{N}$. Then we have $\pi^-(F)=U$ if and only if $\widehat{U}\subset F$.

So suppose $\dimv F=\dimv\widehat{N}$ and $\widehat{U}\subset F$. Then $U=\pi^-(\widehat{U})\subset \pi^-(F)$ and
$$\dim U_i=\dim \widehat{U}_i=\dim F_i=\dim (\pi^-(F))_i,$$
for any $i\in Q_0$, and thus $U=\pi^-(F)$.

Conversely, suppose that $\pi^-(F)=U$ and $F\subset \widehat{M}$.
Since $M$ is a Gorenstein projective module, $\pi^-\widehat{M}=M$, and $\pi^-$ is exact, we get that
$\pi^-(F)\subseteq M$, and then $\pi^-(F)$ is also a Gorenstein projective $\Lambda$-module.
We denote by $F=\oplus F_i$, where $F_i$ are the indecomposable summands of $F$.
$F\subset \widehat{M}$ implies that $F_i$ can be the simple $\Gamma$-module $S_\alpha$ for any $\alpha\in Q_1^{cyc}$.
We denote by $w_i$ be the string of $F_i$. Then $w_i$ must be one of the following forms
\begin{eqnarray*}
&(1)\,\, \iota(\alpha_l\dots \alpha_1);\quad\quad\quad\quad\quad\quad\,\, &\quad (2) \,\,\iota(\beta_m\dots\beta_1\gamma_1^{-1}\dots \gamma_n^{-1});\\
&(3) \,\,\iota(\alpha_l\dots \alpha_1)\alpha^-;\quad\quad\quad\quad \quad  &\quad (4) \,\,(\beta^-)^{-1} \iota(\alpha_l\dots \alpha_1);\\
&(5) \,\,(\beta^-)^{-1} \iota(\alpha_l\dots \alpha_1) \alpha^-;\quad\quad &\quad (6) \,\,(\alpha^-)^{-1}\iota(\beta_m\dots\beta_1\gamma_1^{-1}\dots \gamma_n^{-1});\\
&(7)\,\, \iota(\beta_m\dots\beta_1\gamma_1^{-1}\dots \gamma_n^{-1})\beta^-;\, &\quad (8)\,\, (\alpha^-)^{-1}\iota(\beta_m\dots\beta_1\gamma_1^{-1}\dots \gamma_n^{-1})\beta^-,
\end{eqnarray*}
since $\pi^-(F_i)$ is Gorenstein projective. From the above, we get that there exists a natural injective morphism $h:\Phi(\pi^-(F))\rightarrow F$.
In fact, $\pi^-$ is exact, we get that $\pi^-(\ker h)=0$, since $\pi^-(h)=\Id$. Then $\ker h\cong \bigoplus_{\alpha\in Q_1^{cyc}}S_\alpha^{n_\alpha}$. However, $S_\alpha$ can not be a submodule of $\Phi(\pi^-(F))$, so $\ker h=0$, and then $h$ is injective.
So $\widehat{U}=\Phi(\pi^-(F))$ is a submodule of $F$.
\end{proof}

We can now easily derive the main general geometric properties of the map $\pi_{[N]}$:

\begin{theorem}
For all $M$ and $e$ as above and a generic subrepresentation type $[N]\in\gsub_e(M)$, the following holds:

(i) The varieties $\Gr_{\dimv \widehat{N}}(\widehat{M})$ is smooth with irreducible equidimensional connected components;

(ii) the map $\pi_{[N]}$ is projective;

(iii) the image of $\pi_{[N]}$ is closed in $\Gr_e(M)$ and contains $\overline{\cs_{[N]}}$;

(iv) the map $\pi_{[N]}$ is one-to-one over $\cs_{[N]}$.
\end{theorem}
\begin{proof}
(i) follows from Theorem \ref{theorem smoothness of quiver grassmannians} directly. The map $\pi_{[N]}$ is projective since $\Gr_{\dimv\widehat{N}}(\widehat{M})$ is projective, so (ii) is valid.

For (iii), given a generic embedding $N\subset M$, we also have $\widehat{N}\subset \widehat{M}$ since $\Phi$ preserves injective morphisms. The image of $\pi_{[N]}$ is closed since it is proper. Then
Proposition \ref{proposition fibre of pi} implies that the fibre over $N\subset M$ is non-empty. So the image of $\pi_{[N]}$ contains $\cs_{[N]}$, and then $\overline{\cs_{[N]}}$. In particular, the fibre over a point of $\cs_{[N]}$ reduces to a single point is the special case $U=N$ of Proposition \ref{proposition fibre of pi}.
\end{proof}
\begin{corollary}\label{corollary desingularization map}
For arbitrary $M$ and $e$ as above, the closure $\overline{\cs_{[\widehat{N}]}}$ is an irreducible component of $\Gr_{\dimv \widehat{N}}(\widehat{M})$, and the map
$$\pi=\bigsqcup_{[N]\in\gsub_e(M)}\pi_{[N]}:\bigsqcup_{[N]\in\gsub_e(M)}\overline{\cs_{\widehat{N}}}\rightarrow \Gr_e(M)$$
given by the restrictions of the $\pi_{[N]}$ to $\overline{\cs_{\widehat{N}}}$ is a desingularization of $\Gr_e(M)$.
\end{corollary}
\begin{proof}
The stratum $\cs_{[\widehat{N}]}$ of $\Gr_{\dimv \widehat{N}}(\widehat{M})$ is irreducible and locally closed, of dimension
$$\dim \Hom_\Gamma(\widehat{N},\widehat{M})-\dim \End_\Gamma(\widehat{N})=\dim \Hom_\Gamma(\widehat{N},\widehat{M}/\widehat{N})=\dim \Gr_{\dimv \widehat{N}}(\widehat{M})$$
as in the proof of \cite[Proposition 7.1]{CFR13} since $\Ext^1(\widehat{N},\widehat{N})=0$, so its closure is a connected (irreducible) component of $\Gr_{\dimv\widehat{N}}(\widehat{N})$ and thus a smooth variety.
The image of this component under $\pi_{[N]}$ is thus an irreducible closed subvariety of $\Gr_e(M)$ containing its irreducible component $\overline{\cs_{[N]}}$, and thus it equals $\overline{\cs_{[N]}}$. Together with the other properties of the previous theorem, this implies that $\pi$ is a desingularization.
\end{proof}

\begin{remark}
Following \cite[Remark 7.8]{CFR13}, we conjecture that $\Gr_{\dimv \widehat{N}}(\widehat{M})$ is actually irreducible. This would imply that the constructions of desingularizations of the previous corollaries could be unified to the map
$$\pi= \bigsqcup_{[N]\in\gsub_e(M)}\pi_{[N]}:\bigsqcup_{[N]\in\gsub_e(M)}\Gr_{\dimv \widehat{N}}(\widehat{M}) \rightarrow \Gr_e(M)$$
being a desingularization.
\end{remark}

\begin{corollary}
Let $\Lambda$ be a self-injective gentle algebra. Then the desingularization map $\pi$ defined in Corollary \ref{corollary desingularization map} is defined for any finite generated $\Lambda$-module $M$ and dimension vector $e$.
\end{corollary}
\begin{proof}
Since $\Lambda$ is self-injective, we get that $\mod\Lambda=\Gproj\Lambda$. The statement follows from Corollary \ref{corollary desingularization map} immediately.
\end{proof}

In Section 5, we can see that even for self-injective gentle algebras, their quiver Grassmannians can be singular varieties.

\section{Examples}

\subsection{Self-injective gentle Nakayama algebras}
Let $Q$ be the following left quiver and $\Lambda=KQ/\langle \beta\alpha,\gamma\beta,\alpha\gamma\rangle$. Then its Cohen-Macaulay Auslander algebra $\Gamma=KQ^{Aus}/\langle I^{Aus}\rangle$, where $Q^{Aus}$ is as the following right quiver shows, and the ideal $I^{Aus}=\{ \beta^+\alpha^-,\gamma^+\beta^-,\alpha^+\gamma^- \}$.
\setlength{\unitlength}{1.2mm}
\begin{center}
\begin{picture}(80,30)

\put(10,10){\circle{1.3}}
\put(30,10){\circle{1.3}}
\put(20,20){\circle{1.3}}
\put(11,11){\vector(1,1){8}}
\put(21,19){\vector(1,-1){8}}
\put(29,10){\vector(-1,0){18}}
\put(9,7){$1$}
\put(19,21){$2$}
\put(29,7){$3$}
\put(13,16){$\alpha$}
\put(25,16){$\beta$}
\put(19,7){$\gamma$}

\put(50,10){\circle{1.3}}
\put(50,11){\vector(0,1){8}}
\put(60,25){\circle{1.3}}
\put(60,5){\circle{1.3}}
\put(51,20.5){\vector(2,1){8}}
\put(50,20){\circle{1.3}}
\put(61,24.5){\vector(2,-1){8}}

\put(70,10){\circle{1.3}}
\put(69,9.5){\vector(-2,-1){8}}
\put(70,20){\circle{1.3}}
\put(70,19){\vector(0,-1){8}}
\put(59,5.5){\vector(-2,1){8}}

\put(47,9){\tiny$1$}
\put(47,19){\tiny$\alpha$}
\put(59,26){\tiny$2$}
\put(59,2.5){\tiny$\gamma$}
\put(71,19){\tiny$\beta$}
\put(71,9){\tiny$3$}
\put(50.5,14){\tiny$\alpha^+$}
\put(54,20.5){\tiny$\alpha^-$}
\put(62.5,20){\tiny$\beta^+$}
\put(66.5,14){\tiny$\beta^-$}
\put(63.5,8){\tiny$\gamma^+$}
\put(54.5,8){\tiny$\gamma^-$}

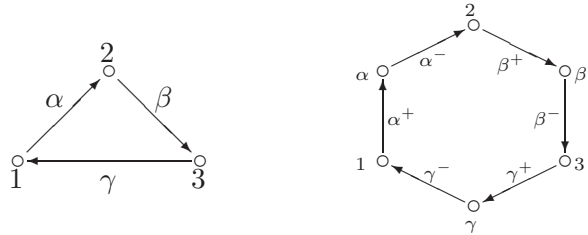
\put(-10,-3){Figure 4. The quiver of $\Lambda$ and its Cohen-Macaulay Auslander algebra.}
\end{picture}
\vspace{0.5cm}
\end{center}

Let $M=P_3\oplus P_3\oplus S_1\oplus S_3$. Let $e=\dimv S_3+2\dimv S_1$. We consider the quiver Grassmannian $X=\Gr_{e}(M)$. For simplicity, we restrict our attention to the full subquiver of $Q$ formed by vertices $1,3$. Choosing appropriate basis, the representation $M$ can be written as
$$\xymatrix{ K^3 && K^3.\ar[ll]_{\tiny\left( \begin{array}{ccc} 1 &0 &0\\
0&1&0\\0&0&0 \end{array}\right)}}$$
In this way, the dimension vector $e$ equals $(2,1)$, thus, identifying $\Gr_2(K^3)$ with $\P^2$, the quiver Grassmannian $X$ can be realized as
$$\{((a_0:a_1:a_2),(b_0:b_1:b_2))\in \P^2\times \P^2| a_0b_0+a_1b_1=0\},$$
which is a singular projective variety of dimension three.

We also restrict our attention to the full subquiver of $Q^{Aus}$ formed by vertices $1,3,\gamma$. The $\Gamma$-module $\widehat{M}=\Phi(M)$ admits the following explicit form:
$$\xymatrix{ K^3 && K^2 \ar[ll]_{\tiny\left( \begin{array}{ccc} 1 &0\\
0&1\\0&0 \end{array}\right)}&& K^3.\ar[ll]_{\tiny\left( \begin{array}{ccc} 1 &0 &0\\
0&1&0 \end{array}\right)} }$$

The only generic subrepresentation type being $N=P_3\oplus S_1$, we thus have to consider subrepresentations of dimension vector
$\dimv \widehat{N}=2\dimv S_1+\dimv S_3+\dimv S_\gamma$ of $\widehat{M}$. Let $Y=\Gr_{\dimv \widehat{N}}(\widehat{M})$. Then the quiver Grassmannian $Y$ can be realized as
$$\{((a_0:a_1:a_2),(c_0:c_1),(b_0:b_1:b_2)\in \P^2\times \P^1\times \P^2|$$
$$a_0c_0+a_1c_1=0,b_1c_0=b_0c_1\},$$
with the desingularization map being the projection to the first and third components.

\subsection{Cluster-tilted algebra of type $\A_n$}
Let $Q$ be the quiver as the following left diagram shows. Let $I=\{\alpha_3\alpha_2,\alpha_2\alpha_1,\alpha_1\alpha_3,\alpha_5\alpha_4,\alpha_6\alpha_5,\alpha_4\alpha_6 \}$.
Then $\Lambda=KQ/\langle I\rangle$ is a cluster-tilted algebra of type $\A_5$.
Let $\Gamma=\Aus(\Gproj\Lambda)$ be the Cohen-Macaulay Auslander algebra. Then its quiver is as the following right diagram shows.

\begin{center}\setlength{\unitlength}{0.7mm}
 \begin{picture}(200,70)
 \put(0,10){\begin{picture}(80,80)
\put(40,0){\circle{2}}
\put(42,0){\vector(1,0){26}}
\put(70,0){\circle{2}}
\put(68,2){\vector(-1,1){12}}
\put(55,15){\circle{2}}
\put(57,17){\vector(1,1){11}}
\put(53,13){\vector(-1,-1){11}}
\put(40,30){\circle{2}}
\put(42,28){\vector(1,-1){12}}
\put(70,30){\circle{2}}
\put(68,30){\vector(-1,0){26}}

\put(39,-5){\small $4$}
\put(69,-5){\small $5$}
\put(54,17){\small $3$}
\put(39,32){\small $2$}
\put(70,32){\small $1$}

\put(54,32){\tiny $\alpha_1$}
\put(54,-3){\tiny $\alpha_5$}
\put(43,9){\tiny $\alpha_4$}
\put(62,9){\tiny $\alpha_6$}
\put(43,21){\tiny $\alpha_2$}
\put(62,20){\tiny $\alpha_3$}
\end{picture}}

\setlength{\unitlength}{0.5mm}
\put(120,0){\begin{picture}(80,80)
\put(40,30){\circle{2}}
\put(40,28){\vector(0,-1){11}}
\put(70,30){\circle{2}}
\put(68,32){\vector(-1,1){12}}
\put(55,45){\circle{2}}
\put(53,43){\vector(-1,-1){11}}
\put(57,47){\vector(1,1){11}}
\put(70,60){\circle{2}}
\put(70,62){\vector(0,1){11}}
\put(70,75){\circle{2}}
\put(68,77){\vector(-1,1){11}}
\put(55,90){\circle{2}}
\put(53,88){\vector(-1,-1){11}}
\put(40,75){\circle{2}}
\put(40,73){\vector(0,-1){11}}
\put(40,60){\circle{2}}
\put(42,58){\vector(1,-1){11}}

\put(40,15){\circle{2}}
\put(42,13){\vector(1,-1){11}}
\put(55,0){\circle{2}}
\put(57,2){\vector(1,1){11}}
\put(70,15){\circle{2}}
\put(70,17){\vector(0,1){11}}

\multiput(53,86)(-1.4,-2.8){10}{$\cdot$}
\multiput(53,86)(1.4,-2.8){10}{$\cdot$}
\multiput(41,58)(3,0){10}{$\cdot$}
\multiput(68,26)(-1.4,-2.8){10}{$\cdot$}
\multiput(41,26)(1.4,-2.8){10}{$\cdot$}
\multiput(41,28)(3,0){10}{$\cdot$}

\put(34,13){\small $4$}
\put(72,13){\small $5$}
\put(72,74){\small $1$}
\put(54,47){\small $3$}
\put(34,73){\small $2$}

\put(72,58){\small 6}
\put(54,92){\small 7}
\put(34,58){\small 8}
\put(34,28){\small 9}
\put(52,-6){\small 10}
\put(71,28){\small 11}
\end{picture}}

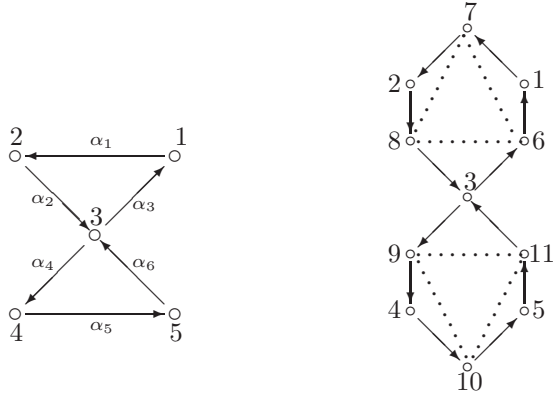
\put(0,-14){Figure 5. Cluster-tilted algebra of type $\A_5$ and its Cohen-Macaulay Auslander algebra.}
\end{picture}
\vspace{0.8cm}
\end{center}

Let $M=P_3\oplus R(\alpha_2)\oplus R(\alpha_6)$. Then $M$ is a Gorenstein projective module. Let $e=\dimv P_3+ \dimv S_1+\dimv S_4$. We consider the quiver Grassmannian $X=\Gr_{e}(M)$. For simplicity, we restrict our attention to the full subquiver of $Q$ formed by vertices $1,3,4$. Choosing appropriate basis, the representation $M$ can be written as
$$\xymatrix{ K^3 && K^3\ar[ll]_{\tiny\left( \begin{array}{ccc} 0 &0 &0\\
0&1&0\\0&0&1 \end{array}\right)} \ar[rr]^{ \tiny\left( \begin{array}{ccc} 1 &0 &0\\
0&0&0\\0&0&1 \end{array}\right)} && K^3 .}$$
In this way, the dimension vector $e$ equals $(2,1,2)$, thus, identifying $\Gr_2(K^3)$ with $\P^2$, the quiver Grassmannian $X$ can be realized as
$$\{((a_0:a_1:a_2),(b_0:b_1:b_2),(c_0:c_1:c_2))\in \P^2\times \P^2\times \P^2| a_1b_1+a_2b_2=0,b_0c_0+b_2c_2=0\},$$
which is a singular projective variety of dimension four.

We also restrict our attention to the full subquiver of $Q^{Aus}$ formed by vertices $1,3,4,6,9$. The $\Gamma$-module $\widehat{M}=\Phi(M)$ admits the following explicit form:
$$\xymatrix{ K^3 && K^2 \ar[ll]_{\tiny\left( \begin{array}{ccc} 0 &0\\
1&0\\0&1 \end{array}\right)}&& K^3\ar[ll]_{\tiny\left( \begin{array}{ccc} 0 &1 &0\\
0&0&1 \end{array}\right)} \ar[rr]^{ \tiny\left( \begin{array}{ccc} 1 &0 &0\\
0&0&1 \end{array}\right)} && K^2\ar[rr]^{ \tiny\left( \begin{array}{ccc} 1 &0 \\
0&0\\0&1 \end{array}\right)} &&K^3.}$$

The only generic subrepresentation type being $N=P_3\oplus S_1\oplus S_4$, we thus have to consider subrepresentations of dimension vector
$\dimv \widehat{N}=2\dimv S_1+2\dimv S_4+\dimv S_3+\dimv S_6+\dimv S_9$ of $\widehat{M}$. Let $Y=\Gr_{\dimv \widehat{N}}(\widehat{M})$. Then the quiver Grassmannian $Y$ can be realized as
$$\{((a_0:a_1:a_2),(d_0:d_1),(b_0:b_1:b_2),(e_0:e_1),(c_0:c_1:c_2))\in \P^2\times \P^1\times \P^2\times \P^1\times\P^2|$$
$$d_0b_2=d_1b_1,b_0e_1=b_2e_0,a_1d_0+d_2d_1=0,c_0e_0+c_2e_1=0\},$$
with the desingularization map being the projection to the first, third and fifth components.

\end{document}